\documentclass[10pt,reqno]{amsart}

\usepackage{amsmath,amsthm,amssymb, latexsym, stackrel,enumerate,bbm, amsfonts, wasysym, tikz}
\usetikzlibrary{patterns}

\usepackage[top=1.1in, bottom=1.2in, left=1.2in, right=1.2in]{geometry}

\usepackage{tikz-cd}

\usepackage{adjustbox}
\usepackage[normalem]{ulem}  
\usepackage[all]{xy}
\usepackage{srcltx}

\usepackage{newclude}
\usepackage{tikzsymbols}
\usepackage{pgf,tikz}

\usepackage{mathtools}
\usepackage{xcolor}
\usetikzlibrary{patterns}

\usepackage{hyperref}
\hypersetup{colorlinks,linkcolor={purple},citecolor={teal},urlcolor={gray}}

\usepackage{graphicx}
\usepackage{tikz-cd}
\usetikzlibrary{shapes}
\usetikzlibrary{cd}
\usetikzlibrary{arrows}
\usetikzlibrary{topaths}
\usetikzlibrary{decorations.pathmorphing}
\usetikzlibrary{shadows.blur}
\usetikzlibrary{intersections}
\usetikzlibrary{calc}
\usepackage{tikz}
\usetikzlibrary{positioning}
\usetikzlibrary{matrix}
\usepackage{float}
\usepackage{subfig}
\usepackage{multicol}

\theoremstyle{plain}
\newtheorem{thm}{Theorem}[section]
\newtheorem{lem}[thm]{Lemma}
\newtheorem{prop}[thm]{Proposition}

\newtheorem{thm-defn}[thm]{Theorem-Definition}

\theoremstyle{definition}
\newtheorem{defn}[thm]{Definition}

\newtheorem{exmp}[thm]{Example}

\theoremstyle{remark}
\newtheorem{rem}[thm]{Remark}

\newcommand{\SG}{A(Q,I, \mathrm{Sp})}
\newcommand{\Qsp}{Q_0^{\times}}

\newcommand{\Hom}{\operatorname{Hom}}

\newcommand{\soc}{\operatorname{soc}}

\newcommand{\orbi}{\mathcal{O}}

\newcommand{\mf}{\mathfrak{m}}
\newcommand{\of}{\mathfrak{o}}

\DeclareMathOperator{\val}{val}

\usepackage{color}

\title{Skew-Brauer graph algebras}

\begin{document}
\author[Garc\'ia-Elsener]{Ana Garc\'ia-Elsener}
\address{\textbf{[AGE]} 
Facultad de Ciencias Exactas y Naturales (CEMIM)\\
Universidad Nacional de Mar del Plata and CONICET. Dean Funes 3350. 
(7600) Mar del Plata, Argentina.}
\email{elsener@mdp.edu.ar}
\author[Guazzelli]{Victoria Guazzelli}
\address{\textbf{[VG]}
Facultad de Ciencias Exactas y Naturales\\
Universidad Nacional de Mar del Plata. 
Dean Funes 3350. 
(7600) Mar del Plata, Argentina.}
\email{vguazzelli@mdp.edu.ar}

\author[Valdivieso]{Yadira Valdivieso}
\address{\textbf{[YV]}
Departamento de Matem\'aticas, Universidad de las Am\'ericas Puebla.
Ex Hacienda Sta. Catarina Mártir S/N. (72810) San Andrés Cholula, Puebla, México.}
\email{yadira.valdivieso@udlap.mx}

\thanks{AGE, VG and YV would like to thank the Isaac Newton Institute for Mathematical Sciences, Cambridge, for support and hospitality during the programme “Cluster Algebras and Representation Theory” where part of the work on this paper was undertaken, this work was supported by EPSRC grant EP/R014604/1. 
YV was supported by EPSRC grant no EP/K032208/1 and the Simons Foundation. 
YV is also supported by the European Union’s Horizon 2020 research and innovation programme under the Marie Sklodowska-Curie grant agreement No H2020-MSCA-IF-2018-838316.
AGE was supported by PICT 2021-I–A-01154 Agencia Nacional de Promoción Científica y Tecnológica and by the EPSRC grant EP/R009325/1. We especially thank Elsa Fernández for her valuable comments on iterated tilted algebras.}

\maketitle
\begin{abstract}
In this work, we introduce a new class of algebras that we denominate skew-Brauer graph algebras. This class contains and generalizes the Brauer graph algebras. We establish that skew-Brauer graph algebras are symmetric and can be defined based on a Brauer graph with additional information.  We show that the class of trivial extensions of  skew-gentle algebras coincides with a subclass of skew-Brauer graph algebras, where the associated skew-Brauer graph has multiplicity function identically equal to one, generalizing a result over gentle algebras. Furthermore, we characterize skew-Brauer algebras of finite representation type. Finally, we provide a geometric interpretation of cut-sets and reflections of algebras using orbifold dissections.
\end{abstract}


\section{Introduction}

Brauer graph algebras are defined  combinatorially through a graph with some additional data on their vertices, given by a multiplicity function and an order over the edges attached to each vertex. They appeared first in representation theory of groups, and were defined in \cite{DF78} to classify indecomposable modules over quasi-Frobenious algebras. 

Brauer graph algebras have been recently related to cluster algebras of surfaces, gentle algebras and mirror symmetry, see for example \cite{MS14, Sch15, OPS18}, moreover they have played an important role in the representation theory of tame symmetric algebras. Most of the relevant questions about their homological properties have been studied, for example, it is well understood their derived equivalence classes \cite{OZ22} \cite{AZ22}, the classification of their two-term tilting complexes \cite{AACh18} and their first Hochschild cohomology group \cite{ChSS20}. A natural relation between trivial extension of gentle algebras and a subclass of Brauer graph algebras has been described in \cite{Sch15}, more precisely, the paper shows that any trivial extension of a gentle algebra is actually a Brauer graph algebra, in which the multiplicity function is identically equal to one.

Several generalizations of Brauer graph algebras have been defined using slightly modified graphs or definitions, in which important properties of the algebra, and its modules, can be directly understood from its graph. See for example \cite{GS17} where the so-called \emph{Brauer  configuration algebras} are defined, or \cite{Sko2006} where the \emph{modified Brauer algebras} are constructed.

In this paper, we define what we call a skew-Brauer graph, which is by definition a Brauer graph $\Gamma$, that has a set of special vertices $\Gamma_0^{\times}$ with valency one and multiplicity function equal to one. Using this extra data we define a path algebra called skew-Brauer graph algebra, in order to link this new family of algebras and the trivial extensions of skew-gentle algebras, which are skew-group algebras of gentle algebras with a $\mathbb{Z}_2$ action.

Skew-gentle algebras were defined in \cite{G99} to study a subclass of matrix-problems, and to describe explicitly their Auslander-Reiten sequences. This family of tame algebras is also related to several subjects, for example, they were used in \cite{GLFS16} to prove that a large class of Jacobian algebras arising from closed surfaces are also tame. Skew-gentle algebras are a generalization of gentle algebras, as they are, by definition, skew-group algebras of gentle. Both families, gentle and skew-gentle algebras share several properties, for example both of them are Gorenstein, see \cite{GR2005}.

Any skew-gentle algebra can also be defined by a triple $(Q,I,\text{Sp})$, where $Q$ is a quiver, $I$ is a subset of paths of length two, and $\text{Sp}$ is a set of loops in $Q$, subject to certain restriction. This triple defines two gentle algebras, one of which we call the underlying gentle algebra. Using this definition, we show that the original skew-gentle algebra and the 
underlying gentle algebra also share some properties.

In this work, we study trivial extensions of skew-gentle algebras, and we show that there is a simple way to define such trivial extension using the trivial extension of the underlying gentle algebra. Even more, we show that any trivial extension of a skew-gentle algebra is a skew-Brauer graph algebra. More precisely, we have the following result.

\smallskip

\noindent\textbf{Theorem A} (Theorem~\ref{thm: skew brauer simetricas}, Theorem~\ref{thm: equiv1} and Theorem~\ref{thm:admissiblecuts} ){\it \,
\begin{enumerate}
    \item[i)] Any skew-Brauer graph algebra is a symmetric finite dimensional algebra.
    \item[ii)] The class of trivial extension of skew-gentle algebras coincides with the subclass of skew-Brauer graph algebras with multiplicity function identically equal to one.
\end{enumerate}}

Deciding whether a trivial extension of an algebra is of finite representation type is a question that has been addressed by several authors. In this work, we consider two different tools that were used to solve that problem. The first one is the admissible cuts or cut sets, defined in \cite{F99, FP06}, to provide a complete characterization of trivial extension of finite representation type. The second one is reflections of quivers and relations, defined in \cite{HW83}, to show that any trivial extension is of finite representation type if and only if it is related to an algebra of Dynkin type through sequences of reflections. In this paper, we study reflections and cut sets that preserve the class of skew-gentle algebras and their trivial extensions, in this context, we have the following results.

\smallskip
\noindent\textbf{Theorem B} (Theorem~\ref{thm: refl-skew} and Theorem~\ref{thm:cuts and reflections}){\it \,
Let $A=\SG$ be a skew-gentle algebra,  $A'$ be its underlying gentle algebra and $i$ be a source in $A$. Then the reflection $S_i$ of $A'$ gives rise to a sequence of reflections  $\overline{S_i}$ on $A$ such that $\overline{S_i}(A)$ is skew-gentle. Furthermore there exists a cut set $\mathcal{D}$ of $T(A)$ such that $\overline{S_i}= T(A)/\langle \mathcal{D}\rangle $}

Following a natural question, we also show the following characterization of skew-Brauer graph algebras of finite representation type, see definition of skew-Brauer tree in Section~\ref{sec:finrep}.

\smallskip
\noindent\textbf{Theorem C} (Theorem~\ref{thm:finite representation type} and Theorem~\ref{thm:skew-brauer tree algebras}){\it \,
Let $B_{\Gamma^{\times}}$ be a skew-Brauer graph algebra.  Then $B_{\Gamma^{\times}}$ is of finite representation type if and only if $\Gamma^{\times}$ is a skew-Brauer tree. Moreover, in this case, $B_{\Gamma^\times}$is isomorphic to the trivial extension of an iterated tilted algebra of type $\mathbb{D}$.}

Additionally, we present a geometric interpretation of cut sets and reflections using orbifold dissections, a geometric model inspired on the one for the derived category of skew-gentle algebras and the Koszul dual algebra, see \cite{AB} and \cite{LSV}. More specifically, we establish the following theorems.

\smallskip
\noindent\textbf{Theorem D} (Theorem~\ref{thm:contracting and adding}){\it \,
Let $A=\SG$ be a skew-gentle algebra, $A^{\mathrm{sg}}$ be its admissible presentation and $(S,M, P, \mathcal O, D^{\times})$ be its orbifold $\color{red}{\times}$-dissection. Then, any good cut set of the trivial extension $T(A^{\mathrm{sg}})$ of $A^{\mathrm{sg}}$ is induced by a sequence of contracting and adding boundary segments in $(S,M, P, \mathcal O, D^{\times})$.}

As a consequence of Theorem B and Theorem D, one can prove that the sequences of reflections  described in Theorem B are also induced by contracting and adding boundary segments.

\noindent\textbf{Remark} 
\begin{enumerate}
    \item Generalizations of Brauer graphs have been studied by several authors. In particular, generalizations related to trivial extension of iterated tilted of type $\mathbb{D}_n$ appeared in \cite{Nor21}, where skew-Brauer trees are named as modified Brauer tree (and double edge algebras).
    \item This work started in November 2021 in the Newton Institute, since then, the authors presented preliminary results in several conferences and workshops \footnote{II Encuentro conjunto RSME-UMA, December 12-16, 2022.
    Website \href{https://www.rsme.es/congresos-y-escuelas/congresos-conjuntos-con-otras-sociedades-cientificas/ii-encuentro-conjunto-rsme-uma/}{https://www.rsme.es/congresos-y-escuelas/congresos-conjuntos-con-otras-sociedades-cientificas/ii-encuentro-conjunto-rsme-uma/}}
    \footnote{Combinatorial aspects of Representation Theory, March 28-31, 2023.
    Website \href{https://wiki.math.ntnu.no/combrep/program}{https://wiki.math.ntnu.no/combrep/program/}}
    \footnote{FD seminar, Talk 133. September 14th, 2023.
    Website \href{https://www.fd-seminar.xyz/talks/2023-09-14/}{https://www.fd-seminar.xyz/talks/2023-09-14/}}.
\item Recently, we found a similar definition of skew-Brauer graph algebras in \cite{Soto24}.
\end{enumerate}

\section{Background}
Let $K$ be an algebraically closed field of characteristic $\neq 2$. Let $Q=(Q_0,Q_1)$  be a finite quiver, where  $Q_0$ denotes the set of vertices and  $Q_1$ denotes the set of arrows. For an arrow $\alpha $ denote $s(\alpha)$ the starting vertex and $t(\alpha)$ the ending vertex. A \textit{path} $p$ of length $l>0$ from the vertex $x$ to the vertex $y$ in $Q$ is an ordered set of arrows $\{\alpha_1, \dots, \alpha_l\}$ such that $s(\alpha_1)=x$, $t(\alpha_l)=y$ and $t(\alpha_i)=s(\alpha_{i+1})$ for all $1 \leq i \leq l-1$, and we define $s(p)=s(\alpha_1)$ and $t(p)=t(\alpha_l)$, and by abuse of notation, we write $p = \alpha_1\alpha_2 \dots \alpha_l$. A path $p$ is a cycle if $s(p)=t(p)$. In addition, for each vertex $i$ we denote by $e_i$ the \emph{stationary path} of length 0. %

Consider the algebra $K Q$ generated by paths of any length, and denote by $J$ the two-sided ideal generated by arrows. A two-sided ideal $I$ of $KQ$ is called \emph{admissible} if there exists $m \geq 2$ such that $J^m \subset I \subset J^2$, where $J^m$ is the ideal generated by paths of length $m$ and $J^2$ is the ideal generated by paths of length 2. The set of  generators in $I$ is called \emph{a set of relations}. 

The pair $(Q,I)$ is said to be a \emph{bound quiver}, and the algebra $KQ/I$ is the bound quiver algebra. Every basic and connected finite dimensional $K$-algebra $A$ can be defined by a bound quiver with an admissible ideal $(Q_A,I_A)$ \cite[Theorem 3.7]{assem2006}, in this case $Q_A$ is called the ordinary quiver of $A$.

Finally a path $p$ in $Q$ is non-zero in $(Q, I)$ if $p\notin I$, and a non-zero path $p$ is \emph{maximal} if for any arrow $\alpha\in Q$, the paths $\alpha p$ and $p\alpha$ are elements of $I$.

\subsection{Gentle and skew-gentle algebras}\label{sec-skew-gentle}
skew-gentle algebras were defined in \cite{G99} as certain skew-group algebras of gentle algebras. 
We first recall the definition of gentle algebra.

\begin{defn}\label{defn:gentle}
A pair $(Q,I)$ given by a quiver $Q$ and an ideal $I$ generated by a set of paths of length two in $Q$ is \emph{gentle} if:

\begin{enumerate}
\item\label{item:def-gentle-at-most-two} there is at most two arrows ending and starting at each vertex $x$;

\item\label{item:def-gentle-at-most-one-in-I} for every arrow $\alpha$ in  $Q_1$ there is at most one arrow $\beta$  such that $t(\alpha) = s(\beta)$ and $\alpha\beta\in I$, and there is at most one arrow $\beta'$ such that  $t(\alpha) = s( \beta')$ and $\alpha\beta' \notin I$;

\item\label{item:def-gentle-at-most-one-outside-I} for every arrow $\alpha$ in $Q_1$ there is at most one arrow $\gamma$ such that $s(\alpha) = t( \gamma)$ and  $\gamma\alpha\in I$, and there is at most one arrow $\gamma'$ such that $s(\alpha) = t( \gamma')$ and $\gamma'\alpha\not\in I$; 

\item\label{item:def-gentle-I-generated-by-paths-of-length-2} the path algebra $A(Q,I)=KQ/I$ is finite dimensional.
\end{enumerate}

An algebra $A$ is \emph{gentle} if it is Morita equivalent to an algebra $A(Q,I)=KQ/I$, where $(Q,I)$ is a gentle pair.

If the bound quiver $(Q, I)$ satisfies \eqref{item:def-gentle-at-most-two}, \eqref{item:def-gentle-at-most-one-in-I}, and \eqref{item:def-gentle-at-most-one-outside-I},  we say that the quotient $KQ/I$ is \emph{locally gentle}.
\end{defn}

\begin{defn}\label{defn:skew-gentle}
A triple $(Q,I,\text{Sp})$ given by a quiver $Q$, an ideal $I$ generated by a set of paths of length two in $Q$, and a subset Sp of loops in $Q$ is \emph{skew-gentle} if the pair $(Q, \langle I\sqcup \{f^2: f\in \text{Sp}\}\rangle )$ is gentle. 

An algebra $A$ is \emph{skew-gentle} if it is Morita equivalent to an algebra $A(Q,I,\text{Sp})=KQ/\langle I\sqcup\{f^2-f: f\in \text{Sp}\}\rangle$, where $(Q,I, \text{Sp})$ is a skew-gentle triple.

We refer to the elements of Sp as \emph{special loops}, and we denote by $Q_0^{\times}$  the set of \emph{special vertices} of $Q$ where a special loop is attached.
\end{defn}

It is clear that, in general, the ideal associated with a skew-gentle algebra $\SG$ is not admissible. However, in \cite{G99}, the authors provide an admissible presentation of $\SG$. The reader can see this construction in Example \ref{ex: toy example}.

\subsection{Trivial extensions}\label{sec-trivial-extensions}

Trivial extensions of artin algebras play an important role in representation theory. In this section we describe a construction for the quiver and relations of trivial extension of a bound quiver algebra. We mainly follow the articles \cite{FP02,FSTTV22,Sch15}.

The \emph{trivial extension of an algebra $A$} is the algebra $T(A)=A\ltimes DA$ with underlying vector space $A\oplus DA$ and with the multiplication defined by $(a,f)(b,g)=(ab, ag+fb)$, for $a, b\in A$ and $f,g\in DA$, where $D=\Hom_K(-, K)$.

For every finite dimensional algebra $A = KQ_A/I_A$, where $I_A$ is an admissible ideal, we can construct its trivial extension $T(A)$ based on the quiver and relations of $A$ and a basis of the two-sided ideal $\soc_{A^e}A$, where $A^e$ is the enveloping algebra. 

Recall that, as a $K$-vector space, $\soc_{A^e}A$ can be defined as the intersection of
\[ \soc (_AA) = \{a\in A \colon Ja =0 \} \ \ \ \mathrm{and} \ \ \ \soc (A_A) = \{a\in A \colon aJ =0 \}, \]
where $J$ denotes the Jacobson radical, i.e. the ideal generated by arrows. See \cite[Exercise 4.20]{lam1991} and 
\cite[Lemma 2.10]{assem2006}.

Let  $\mathcal M=\{p_1, \dots, p_r\}$ be a $K$-basis of $\soc_{A^e}A$. The ordinary quiver of $T(A)$ is given by:
\begin{enumerate}
\item $(Q_{T(A)})_0=(Q_A)_0$
\item $(Q_{T(A)})_1=(Q_A)_1 \cup \{\beta_{p_1},\dots, \beta_{p_r}\}$, where $\beta_{p_i}$ is an arrow from $t(p_i)$ to $s(p_i)$ for each $p_i\in \mathcal{M}$, with $1\leq i\leq r$.
\end{enumerate}

To give an explicit description of the ideal of relations $I_{T(A)}$ we need to recall some definitions from \cite[Section 3]{FP02}.

\begin{defn}\label{def:elementary}
Let $\mathcal{M}=\{p_1, \dots, p_r\}$ be a basis of $\soc_{A^e} A$. Let
$C=\alpha_1\dots\alpha_l\beta_{p_j}$ be an oriented cycle in $KQ_{T(A)}$ such that $p_j^*(\alpha_1\dots\alpha_l)\neq 0$ where $\alpha_1
\dots \alpha_l$ is a path in $KQ_A$, $p_j\in \mathcal{M}$, and $p_j^*$ is the dual of $p_j$. The cycle $C$ and all the cyclic permutations of $C$ are called \textit{elementary}.
In this case, we say that $\omega(C)=p_j^*(\alpha_1\dots\alpha_l)\in K$ is the \textit{weight} of $C$.
\end{defn}

Let $C$ be a cycle of a quiver $Q$ and $q$ be a path contained in some cyclic permutation of $C$. The \emph{supplement of $q$} in $C$ is the (possibly stationary) path $p$  such that $C=pq$ up to cyclic equivalence.

For each vertex $x\in Q_{T(A)}$ we define the two-sided ideal $I_x'$  in $Q_{T(A)}$ generated by

\begin{enumerate}
\item Oriented cycles from $x$ to $x$ which are not contained in any elementary cycle.
\item The elements of the form $\omega (C')C-\omega (C)C'$, where $C$ and $C'$ are different elementary cycles starting and ending at $x$.
\end{enumerate}

The ideal of relations for the trivial extension of a finite dimensional algebra was described in \cite{FSTTV22}. 

\begin{thm}\cite[Theorem A.1]{FSTTV22}\label{ideal TA}
    Let $A=KQ_A/I_A$ a finite dimensional algebra and let $T(A)=KQ_{T(A)}/I_{T(A)}$ be its trivial extension. Then, the quiver $Q_{T(A)}$
is as above
and the ideal $I_{T(A)}$ is generated by the union of the following sets.

\begin{enumerate}
    \item A generating set of the ideal of relations $I_A$ of $A$.
    \item The paths that are not contained in (cyclic permutations of) any elementary cycle.
    \item Linear combination of paths $\rho  \in e_xKQ_{T(A)}e_y$ such that $\rho q\in I_x'$  or $q\rho \in I_{y}'$  for any supplement path $q$ in an elementary cycle $C$, where $I_x'$ is the ideal defined above.
\end{enumerate}
\end{thm}

\subsection{Repetitive algebra and reflections}\label{sec:repetitive}

In this section, we recall the definition of the repetitive algebra and reflections, both of them linked to an alternative construction for a trivial extension of an algebra. We mainly follow \cite{JSch99, FP02, HW83}.

Let $A=KQ/ I$ be a finite dimensional algebra with $I$ be an ideal generated by a set of relations $\rho$ for $Q$ which are either zero relations or commutativity relations, and  let $\mathcal M$ be a set of maximal paths in $(Q,I)$. The repetitive algebra $\widehat A$ is defined as the path algebra $K\widehat{Q}/ \langle \widehat \rho\rangle$, where $\widehat Q$ and $\widehat \rho$ are defined as follows.

\begin{enumerate}
\item The set of vertices of $\widehat{Q}$ is the set $\{x[n]: x\in Q_0 \text{ and } n\in \mathbb{Z}\}$.
\item In the same way, every arrow $\alpha:x\rightarrow y$ in $Q_1$ induces arrows $\alpha[n]:x[n]\rightarrow y[n]$ in $\widehat{Q}$ for all $n\in \mathbb{Z}$. 
\item For each maximal path $p\in \mathcal{M}$ in $(Q,I)$, with $x=s(p)$ and $y=t(p)$, there is an  arrow $\overline{p}[n]:y[n]\rightarrow x[n+1]$ in $\widehat{Q}$, for all  $n\in \mathbb{Z}$. These arrows are called \emph{connecting arrows}. 
\end{enumerate}

To define the set of relations $\widehat\rho$, we need to introduce some definitions for paths in $(Q, I)$. 
Let $p$ be a path in $(Q,I)$, we denote by $p[n]$ the corresponding path in $\widehat{Q}$. If $p=p_1p_2$ is a maximal path in $(Q,I)$, then $p_2[n]\overline{p}[n]p_1[n+1]$ is called  \emph{full path} in $\widehat{Q}$.

The set of relations $\widehat{\rho}$ of $\widehat{Q}$ are given by the following list.
\begin{itemize}
    \item For each monomial relation $p\in \rho$, the path $p[n]$ is in $\widehat{\rho}$, for all  $n\in \mathbb{Z}$.
    \item For each commutative relation $p_1-p_2\in \rho$, then $p_1[n]-p_2[n]$ is in $\widehat{\rho}$, for all  $n\in \mathbb{Z}$.
    \item Let $p$ be a path in $\widehat{Q}$ which contains a connecting arrow. If $p$ is not a sub-path of a full path, then $p\in \widehat{\rho}$.
    \item If $p=p_1p_2p_3$ and $q=q_1q_2q_3$ are maximal paths in $(Q,I)$  such that $p_2=q_2$, then $p_3[n]\overline{p}[n]p_1[n+1]-q_3[n]\overline{q}[n]q_1[n+1]$ is in $\widehat{\rho}$, for all  $n\in \mathbb{Z}$.
\end{itemize}

To define reflections over an algebra, we need to consider $\nu$ the Nakayama automorphism of $\widehat A$, such that $\nu (x[n])=x[n+1]$ for every $x\in Q_0$ and every $n\in\mathbb{Z}$.

Let $x$ be a source of $Q_A$, we denote by ${\sigma_x^-Q_A}$ the full
sub-quiver of $\widehat{Q_{{A}}}$ 
with vertex set \[S=\{y[0]\in (\widehat{Q_{{A}}})_0\mid y\in (Q_A)_0,y\neq x\}\cup \{x[1]\}.\]

We note that ${\sigma_x^-Q_A}$ is a \emph{complete $\nu$-slice}, which is defined as a connected full subquiver
of $\widehat{Q_{{A}}}$ that contains exactly one vertex from each
$\nu$-orbit of vertices of $Q_{\widehat{A}}$.

The \textbf{negative reflection} $S_x^-(A)$ of $A$ is the algebra $\operatorname{End}_{\widehat{A}}
\left(\bigoplus\limits_{x[z]\in S} e_{x[z]}\widehat{A}\right)$, where $\{e_{x[z]}\}_{x[z]\in S}$ is a set of primitive orthogonal idempotents of $\widehat{A}$.

A \textbf{negative reflection sequence} of $Q_A$ is a sequence of vertices $x_1, \dots, x_r$ of $Q_A$ such that $x_j$ is a source of $\sigma^-_{x_{j-1}}\dots\sigma^-_{x_1}Q_A$ for $1\leq j\leq r$. 

Similarly, given a sink $x$ of $Q_A$, a \emph{positive reflection} of an algebra is defined, considering $\sigma^+_xQ_A$ as the full-subquiver of $\widehat{Q}_A$ with vertex set  \[S=\{y[0]\in (\widehat{Q_{{A}}})_0\mid y\in (Q_A)_0,y\neq x\}\cup \{x[-1]\}.\]

\begin{rem}
    The quiver of the negative (positive) reflection $S^{-}_x(A)$ (resp. $S^{+}_x(A)$) coincides with the quiver $\sigma^-_xQ_A$ (resp. $\sigma^+_xQ_A$).
\end{rem}

\section{Trivial extension of a skew-gentle algebra} \label{Sec:trivial extensions of skew}

In this section, we describe the quiver and relations for the trivial extension of the admissible presentation of a skew-gentle algebra. To do this, we introduce a construction of quivers and relations that depends on a set of vertices and cycles from a given quiver. By abuse of notation and convenience, we call these vertices \emph{special}, and we denote them by $\Qsp$ as in the previous section. This construction is slightly different to the one presented in \cite{G99}, and revisited in \cite{AB,LSV}, implemented to define the admissible presentation of any skew-gentle algebra.

\begin{defn}\label{definition-admissible quiver} (sg-quiver $Q^{\mathrm{sg}}$) Let $Q$ be a quiver with a set of vertices $Q_0^{\times}$ (called \emph{special vertices}), such that there is no loop incident at any special vertex. We define the sg\emph{-quiver} $Q^{\mathrm{sg}}$ as follows.  We start by duplicating the special vertices in $Q$, that is, for each vertex $x$ in $Q_0$ we define
\[Q^{\text{sg}}_0(x) = \left\{ 
\begin{array}{ll}
\{x^+, x^-\} & \mbox{if $x \in \Qsp$} \\
\{ x \} & \mbox{otherwise,} \\
\end{array}\right.
\]
obtaining the set of vertices $Q^{\text{sg}}_0:= \bigcup_{x\in Q_0} Q^{\text{sg}}_0(x)$. For each arrow $\alpha$ in $Q_1$ we define the set of arrows
\[Q^{\text{sg}}_1(\alpha):=\{(x,\alpha, y) \mid x \in Q_0^{\text{sg}}(s(\alpha)), y \in Q_0^{\text{sg}}(t(\alpha))\}\]
obtaining the set of arrows $Q_1^{\text{sg}}:= \bigcup_{\alpha\in Q_1}Q_{1}^{\text{sg}}(\alpha)$.
\end{defn}

In other words, while arrows incident to  non-special vertices remain the same, for an arrow in $Q$ of the form $x \xrightarrow{\alpha} y$, if one of the vertices, say $y$, belongs to $\Qsp$ then we will have two arrows in $Q^{\text{sg}}$. If both vertices $x$ and $y$ are special there will be four arrows in $Q^{\text{sg}}$. See these cases below. 

$$
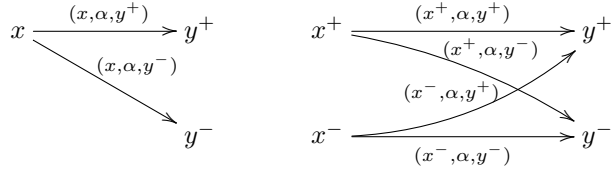
\xymatrix{x\ar[rr]^{(x,\alpha, y^+)} \ar[rrd]^{(x, \alpha, y^-)} && y^+ \\ && y^-} \hspace{30pt} \xymatrix{x^+ \ar[rrr]^{(x^+,\alpha,y^+)} \ar@/^0.55pc/[rrrd]^{(x^+, \alpha, y^-)} &&& y^+ 
  \\ x^-\ar@/_0.9pc/[rrru]^{(x^-,\alpha, y^+)}\ar[rrr]_{(x^-,\alpha, y^-)} &&& y^- }.$$
To shorten notation, we use upper scripts for new arrows in $Q^{\text{sg}}$, for instance, $( x, \alpha, y^+ )$ can be denoted $ x \xrightarrow{\alpha^+} y^+ $, and $( x^-, \alpha, y^+ )$ can be denoted $ x^- \xrightarrow{^-\alpha^+} y^+ $, see Example~\ref{ex: toy example}.  In our proofs we denote arrows after special vertex duplication by $ \ ^{\varepsilon}\alpha^{\varepsilon'}$, where the indices $\varepsilon, \varepsilon'$ will be always considered in $\{ +,-, \emptyset \}$.

The next definition generalizes the construction of the admissible ideal of a skew-gentle algebra.

\begin{defn}\label{dfn:sg-ideal} (sg-ideal $I^{\textrm{sg}}$)
Let $A=KQ/I$ be a finite dimensional algebra with $\Qsp$ a set of  special vertices, as in Definition~\ref{definition-admissible quiver}, and $\mathcal{C}$ be a (possibly empty) set of cycles of length at least two in $Q$, called \emph{special cycles}, such that:

\begin{itemize}
    \item The class $\mathcal{C}$ is closed under cyclic permutations.
    \item There is at most one special cycle starting at each vertex in $\Qsp$.
    \item The ideal $I$ is generated by the union of the following (possible empty) sets:
    \begin{itemize}
        \item A set of paths of length at least two that contains no paths of the form $\alpha\beta$ with $x=t(\alpha)=s(\beta) \in \Qsp$.
    \item All commutative relations of the form $C_1-C_2$, where \( C_1 \) and \( C_2 \) are distinct cycles of $\mathcal{C}$ starting at the same non-special vertex.
    \end{itemize}
\end{itemize}

Let $Q^{\text{sg}}$ be the sg-quiver of $Q$. We define the sg\emph{-ideal $I^{\mathrm{sg}}$} of $KQ^{\text{sg}}$ as the ideal generated by the union of the following sets:

\begin{itemize}
    \item[\textbf{Type a:}] Commutative relations $({}^\varepsilon \alpha ^{+})({}^+\beta ^{\varepsilon '})- ({}^\varepsilon \alpha ^{-})({}^-\beta ^{\varepsilon '})$ such that $\alpha\beta$ is a path of length 2 in $Q$ with $t(\alpha) \in$ $\Qsp$, for all possible $\varepsilon, \varepsilon'\in \{+,-, \emptyset\}$.
    \item[\textbf{Type b:}]
     Commutative
    relations  $(\alpha_1^{\varepsilon_2})({}^{\varepsilon_2}\alpha_2^{\varepsilon_3}) \dots ({}^{\varepsilon_k}\alpha_k)-(\beta_1^{\varepsilon'_2})({}^{\varepsilon'_2}\beta_2^{\varepsilon'_3}) \dots ({}^{\varepsilon'_l}\beta_l)$, for each difference of two  different special cycles $C_1-C_2=\alpha_1\dots\alpha_{k}-\beta_1\dots\beta_{l}$ starting at the same  non-special vertex,  for all possible $\varepsilon_i,\,\varepsilon'_j=\{+,-, \emptyset\}$ with $i=1, \dots, k$ and $j=1, \dots, l$. 
    \item[\textbf{Type c}:] Monomial relations $({}^{\varepsilon_1}\alpha_1^{\varepsilon_2})({}^{\varepsilon_2}\alpha_2^{\varepsilon_3}) \dots ({}^{\varepsilon_k}\alpha_k^{\varepsilon_{k+1}})$ where $\alpha_1\dots\alpha_k$ is a monomial relation of $I$, for all possible $\varepsilon_i=\{+,-, \emptyset\}$ and $i=1, \dots, k+1$.
    \item[\textbf{Type d:}]
    A monomial relation $({}^{\varepsilon_1}\alpha_1^{\varepsilon_2})({}^{\varepsilon_2}\alpha_2^{\varepsilon_3}) \dots ({}^{\varepsilon_k}\alpha_k^{\varepsilon_{k+1}})$, for each special cycle $C=\alpha_1\dots\alpha_{k}$ starting at a special vertex, such that $\varepsilon_1\neq \varepsilon_{k+1}$, for all possible $\varepsilon_i=\{+,-, \emptyset\}$ and $i=1, \dots, k+1$.
\end{itemize}

We say that $A^{\text{sg}}=KQ^{\text{sg}}/I^{\text{sg}}$ is the sg\emph{-bound quiver algebra} of $A$.
\end{defn}

To show that the previous definitions generalize the construction of the admissible presentation of a skew-gentle algebra, we need to introduce the following definition.

\begin{defn}\label{def-auxiliary algebra} Let $A=\SG$ be skew-gentle algebra, with $\Qsp$ the set of special vertices in $Q$, and let $\rho$ be the set of paths of length two that generates $I$. The  \emph{auxiliary gentle algebra} $A'$ is the algebra $KQ'/ I'$ where $Q'=Q\setminus \text{Sp}$ and $I'= \langle \rho \setminus \{\alpha\beta\in I: x=t(\alpha)=s(\beta) \in \Qsp \}\rangle $.  
We say that $Q'$ is the \emph{auxiliary quiver} of $A$ and $I'$ is the \emph{auxiliary ideal} of $A$. 
\end{defn}

\begin{rem}\label{rmk:change of quivers}
With the above notation, the admissible presentation of a skew-gentle algebra $A=\SG$, coincides with the sg-bound quiver algebra $(A')^{\text{sg}}$ of the auxiliary gentle algebra of $A$. In this case, when we apply  Definition~\ref{definition-admissible quiver} and Definition~\ref{dfn:sg-ideal}, we consider the set of special cycles to be empty. To clarify the construction, please see Example~\ref{ex: toy example}. In this context, by abuse of notation, we denote $(Q')^{\text{sg}}$ by $Q^{\text{sg}}$, $(I')^{\text{sg}}$ by $I^{\text{sg}}$, and we say that $A^{\text{sg}} = K Q^{\text{sg}}/I^{\text{sg}}$ is the admissible presentation for the skew-gentle algebra $A$.
\end{rem}

\begin{exmp}\label{ex: toy example}  
Let $A=\SG$ be the skew-gentle algebra with quiver $Q$: 
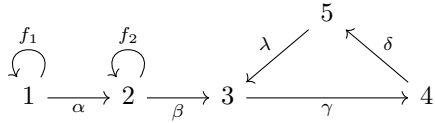
\[\begin{tikzcd}
 &  &  & 5\arrow[ld,"\lambda", swap] &  \\
1 \arrow[r,swap,"\alpha"] \arrow[out=60,in=120,loop  , distance=1.6em,swap, "f_1"] & 2 \arrow[r,swap, "\beta"]\arrow[out=60,in=120,loop  , distance=1.6em, swap, "f_2"] & 3 \arrow[rr,swap, "\gamma"] &  & 4, \arrow[lu, "\delta",swap]
\end{tikzcd}\]
ideal $I = \langle \alpha \beta, \gamma \delta, \lambda \gamma \rangle$ and $\mathrm{Sp}=\{f_1, f_2\}$ the set of special loops. Following Remark~\ref{rmk:change of quivers} to compute the quiver of the admissible presentation of $A$, we consider $A'=KQ'/I$ be the auxiliary gentle algebra, where $Q'=Q\setminus\{f_1, f_2\}$, $I'= \langle \gamma \delta, \lambda \gamma \rangle$,  $\Qsp=\{1,2\}$, and the set of special cycles $\mathcal{C}$ is the empty set.

Following the notation, the sg-bound quiver algebra $(A')^{\text{sg}}$ of $A'=KQ'/I'$ is the algebra $KQ^{\text{sg}}/I^{\text{sg}}$, where $Q^{\text{sg}}$ is the quiver

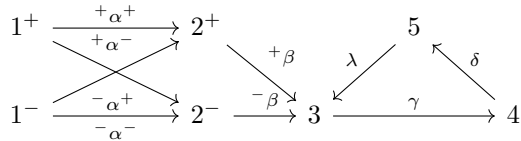
\begin{figure}[h!]
   \adjustbox{scale=1,center}
{\begin{tikzcd}
1^+ \arrow[rr, "^{+}\alpha^+"] \arrow[rrd, "^{+}\alpha^-" near start]   &  & 2^+ \arrow[rd, "^{+}\beta"] &                        & 5 \arrow[ld, "\lambda"'] &                         \\
1^- \arrow[rr, "^{-}\alpha^-"'] \arrow[rru, "^{-}\alpha^+"' near start] &  & 2^- \arrow[r, "^{-}\beta"]  & 3 \arrow[rr, "\gamma"] &                          & 4 \arrow[lu, "\delta"']
\end{tikzcd}}
\end{figure}
and $I^{\text{sg}}=\langle (^+\alpha^+)(^+\beta)-(^+\alpha^-)(^-\beta), (^-\alpha^+)(^+\beta)- (^+\alpha^-)(^-\beta), \gamma \delta, \lambda \gamma \rangle$, which coincides with the admissible presentation of $A$ as defined in \cite{G99} and denoted by the authors as $A^{\text{sg}}$.
\end{exmp}

\begin{rem}\label{rem:comm relations}
    Let $p=\alpha_1 \dots \alpha_r$ be a non-zero  path from $x$ to $y$ in a bound quiver algebra $KQ/I$. Then, $p$ induces 
 a set of non-zero paths $({}^{\varepsilon_1}\alpha_1^{\varepsilon_2}) \cdots ({}^{\varepsilon_r}\alpha_r^{\varepsilon_{r+1}})$  in the sg-bound quiver algebra $KQ^{\text{sg}}/I^{\text{sg}}$, for all possible $\varepsilon_i=\{+, -,\emptyset\}$, with $i=1,\dots,r+1$. 
 Notice that by the commutativity relations of Type (a) of the sg-ideal $I^{\text{sg}}$  given in Definition \ref{dfn:sg-ideal}, any two paths $({}^{\varepsilon_1}\alpha_1^{\varepsilon'_2}) \cdots ({}^{\varepsilon'_r}\alpha_r^{\varepsilon_{r+1}})$ and $({}^{\varepsilon_1}\alpha_1^{\varepsilon_2}) \cdots ({}^{\varepsilon_r}\alpha_r^{\varepsilon_{r+1}})$ from $x^{\varepsilon_1}$ to $y^{\varepsilon_{r+1}}$ are equivalent modulo $I^{\text{sg}}$. Then, we denote the unique representative path
in $KQ^{\text{sg}}/I^{\text{sg}}$ by ${}^{\varepsilon_1} p ^{\varepsilon_{r+1}}$. 
\end{rem}

In order to  compute the quiver of the trivial extension of a skew-gentle algebra $A=\SG$, we need to consider a basis of the socle of the admissible presentation $A^{\text{sg}}$ of $A$ (see Section \ref{sec-trivial-extensions}). 
Recall that a path $p$ (in an admissible presentation $A=KQ/I$) is called \emph{maximal}, if $p\alpha=0$ and $\alpha p=0$ in $A$ for any arrow $\alpha$ in $Q$.

\begin{lem}\label{lem: paths Asg-A'}
    Let $\SG$ be a skew-gentle algebra and let 
    $A'=KQ'/I'$ be its auxiliary gentle algebra. Consider $A^{\mathrm{sg}}=KQ^{\mathrm{sg}}/I^{\mathrm{sg}}$ to be the admissible presentation of $A$.   
    The following statements hold.
    \begin{enumerate}
        \item[(a)] A non-zero path in $A'$ induces at most four linearly independent non-zero paths in $A^{\mathrm{sg}}$.
        Conversely, any non-zero path in $A^{\mathrm{sg}}$ is induced by a path in $A'$.
        \item[(b)] A path in $A'$ is maximal if and only if any induced path in $A^{\mathrm{sg}}$ is a maximal path.
    \end{enumerate}
\end{lem}

\begin{proof}
 Let $A=\SG$ be a skew-gentle algebra with $\Qsp$ the set of special vertices. We consider $A^{\text{sg}}$ the admissible presentation of $A$ and $A'$ its auxiliary gentle algebra.

 $(a)$ Let $p=\alpha_1 \dots \alpha_r$ be a non-zero  path  in $A'$. 
 Following Remark \ref{rem:comm relations}, $p$ induces
 set of paths ${}^{\varepsilon} p ^{\varepsilon'}$
 { for all possible $\varepsilon,\varepsilon'=\{+, -,\emptyset\}$}. In particular, a path $p$ in $A'$ induces exactly:

\begin{itemize}
\item four different equivalence classes of  paths in $A^{\text{sg}}$ if and only if $s(p)$ and $t(p)$ are special vertices, { each one with different starting and ending vertex}; 
\item two different equivalence classes of  paths in $A^{\text{sg}}$ if and only if one and only one of the vertices $s(p)$ or $t(p)$ is a special vertex, coinciding exclusively either at the starting or a the ending vertex;
\item exactly one equivalence class of  paths in $A^{\text{sg}}$ if and only if none of the vertices $s(p)$ and $t(p)$ are special.
\end{itemize}

It is also clear that any non-zero path $({}^{\varepsilon_1}\alpha_1^{\varepsilon_2}) \cdots ({}^{\varepsilon_r}\alpha_r^{\varepsilon_{r+1}})$ in $A^{\text{sg}}$ is induced by the non-zero path $p=\alpha_1 \dots \alpha_r$ in $A'$.  
 \vspace{.1in}
    
$(b)$ Let $p=\alpha_1 \dots \alpha_r$ be a maximal path in $A'$, and let  ${}^{\varepsilon} p ^{\varepsilon'}$ be an induced path in $A^{\text{sg}}$, with $\varepsilon,\varepsilon'\in \{+,-,\emptyset\}$. Suppose that ${}^{\varepsilon} p ^{\varepsilon'}$ is not maximal in $A^{\text{sg}},$ then there exist an arrow ${}^{\varepsilon_1} \beta^{\varepsilon}$ (or ${}^{\varepsilon'} \gamma^{\varepsilon_2}$) such that the path $({}^{\varepsilon_1} \beta^{\varepsilon})({}^{\varepsilon} p ^{\varepsilon'})\neq 0$ (or $({}^{\varepsilon} p ^{\varepsilon'})({}^{\varepsilon'} \gamma^{\varepsilon_2})\neq 0$, respectively). Therefore, such a path in $A^{\text{sg}}$ is induced by a non-zero path $\beta p$ (or $p\gamma$, respectively) in $A'$, contradicting the fact that $p$ is a maximal path. 
With similar argument, we claim that any maximal path ${}^{\varepsilon} p ^{\varepsilon'}$ in $A^{\text{sg}}$ with all possible $\varepsilon,\varepsilon'\in \{+,-,\emptyset\}$ is induced by a maximal path $p$ in $A'.$\end{proof}

Recall that by \cite[Section 2.1]{Sch15}, when $A$ is gentle the set of maximal paths forms a basis for $\operatorname{soc}_{A^e}(A)$.

\begin{lem}\label{lemma:basis of soc A}
Let $A=\SG$ be a skew-gentle algebra and $A^{\mathrm{sg}}=KQ^{\mathrm{sg}}/I^{\mathrm{sg}}$ be its admissible presentation. Then the set of maximal paths $\mathcal{M}$ in $A^{\mathrm{sg}}$ forms a basis of $\operatorname{soc}_{(A^{\mathrm{sg}})^e}(A^{\mathrm{sg}})$.
\end{lem}

\begin{proof}
Let $A'$ be the auxiliary gentle algebra of $A$. Denote $\operatorname{soc}_{(A^{\text{sg}})^e}(A^{\text{sg}})$ by $R$. As a $K$-vector space, $R$ 
is the direct sum \[R=\bigoplus_{x,y\in {Q^{\text{sg}}_0}}e_xRe_y  \]
where $e_x$ and $e_y$ are the trivial paths at a vertex $x$ and $y$ in $Q^{\text{sg}}_0$. 

Let $x$ and $y$ be two vertices in $Q^{\text{sg}}$, and let $x'$ and $y'$ be their associated vertices in $A'$. Since $A'$ is gentle, there are at most two linearly independent classes of paths $p'=\alpha_1\dots \alpha_n$ and $q'=\beta_1\dots \beta_m$ from $x'$ to $y'$ (see \cite[Lemma 3]{BMM03}). By Lemma~\ref{lem: paths Asg-A'} and Remark~\ref{rem:comm relations}, there are at most two linearly independent classes of paths $p$ and $q$ from $x$ to $y$.

Note that there are at most four arrows ending at $x$ and at most four arrows starting at $y$, which correspond to arrows  ending at $x'$ and arrows starting at $y'$. We will show that the paths $p$ and $q$ induce a basis for $e_x Re_y$. We need to analyse each case based on the number of arrows ending at $x$ and starting at $y$, respectively, and if there is exactly one independent class of paths from $x$ and $y$ or two, but we will describe two of them.

Suppose there are two independent classes of paths from $x$ to $y$ and suppose that there is exactly four arrows ending at $x$ and four arrows staring at $y$, then there are exactly two arrows $\gamma_1$ and $\gamma_2$ ending at $x'$ and two arrows  $\delta_1$ and $\delta_2$ staring at $y'$, and $x'$ and $y'$ are not special vertices of $A'$.

Since $A'$ is gentle, either $\gamma_1\alpha_1\notin I$ or $\gamma_2\alpha_1\notin I$, and either $\gamma_1\beta_1\notin I$ or $\gamma_2\beta_1\notin I$. Therefore, $p'$ and $q'$ are not maximal in $A'$, and by Lemma~\ref{lem: paths Asg-A'}, neither are $p$ and $q$ in $A^{\text{sg}}$, and as a consequence $e_xRe_y=\{0\}$.

Suppose now that there are exactly two arrows $\gamma_1$ and $\gamma_2$ ending at $x$ and no arrows starting at $y$. In this case, we need to consider two cases, $x'$ is a special vertex or not. In the first case, we have that $\gamma_1$ and $\gamma_2$ come from an arrow $\gamma$ in $A'$, and there is exactly one independent class of path $p'$ from $x'$ to $y'$, since there is exactly one arrow staring at $x'$. Hence $\gamma \alpha_1\notin I$, and $p'$ is not maximal in $A'$, neither $p$ in $A^{\text{sg}}$, then $e_xRe_y=\{0\}$.

If $x'$ is not a special vertex, we have that either $\gamma_1$ and $\gamma_2$ are arrows of $A'$, or they come from an arrow $\gamma$ in $A'$. In the first case, by definition of gentle and a similar analysis as in the first case, the path or paths from $x'$ to $y'$ are not maximal, and neither the associated path or paths in $A^{\text{sg}}$.

In the second case, if $p\neq q$, then either $\gamma\alpha_1\in I$ or $\gamma\beta_1\in I$; hence, $p$ is maximal or $q$ is maximal respectively. As a consequence $\{p\}$ or $\{q\}$ is a basis of $e_xRe_y$.

Finally, if there is exactly one independent class of paths, namely $p$, then either $\gamma\alpha_1\in I$ or $\gamma\alpha_1\notin I$. In the first situation, $p$ is maximal and is a basis of  $e_xRe_y$, while in the second case, $p$ is not maximal and therefore $e_xRe_y=\{0\}$.\end{proof}

\begin{lem}\label{lemma:quiver of trivial extension}
Let $A=\SG$ be a skew-gentle algebra,  $A^{\mathrm{sg}}=KQ^{\mathrm{sg}}/I^{\mathrm{sg}}$ its admissible presentation and $A'=KQ'/I'$ its auxiliary gentle algebra. Then the quiver $Q_{T(A^{\mathrm{sg}})}$ of the trivial extension of $A^{\mathrm{sg}}$ coincides with the $\mathrm{sg}$-quiver $Q^{\mathrm{sg}}_{T(A')}$ of the trivial extension of $A'$ with special vertices $\Qsp$.
\end{lem}

\begin{proof}
It follows by Lemma \ref{lemma:basis of soc A} and by Lemma \ref{lem: paths Asg-A'} (b).
\end{proof}

\begin{lem}\label{lem: elementary cycles}
    Let $A=\SG$ be a skew-gentle algebra and $A^{\mathrm{sg}}=KQ^{\mathrm{sg}}/I^{\mathrm{sg}}$ its admissible presentation and let 
    $A'=KQ'/I'$ be its auxiliary gentle algebra. We denote by $T(A^{\mathrm{sg}})$ and $T(A')$ the trivial extension of $A^{\mathrm{sg}}$ and $A'$, respectively. Then
   a cycle $C$ in $T(A')$ is elementary if and only if any induced cycle ${}^{\varepsilon}C^{\varepsilon}$ in $Q^{\mathrm{sg}}_{T(A')}$ is an elementary cycle in $T(A^{\mathrm{sg}})$, for all possible $\varepsilon\in \{+,-,\emptyset\}$.
\end{lem}

\begin{proof}
Let $T(A')$ be the trivial 
extension of $A'$, and let $C$ be an elementary cycle in $T(A')$. Then, $C=\rho_1\beta_{\rho}\rho_2$, where $\rho=\rho_2\rho_1$ is a maximal path in $A'$ and $\beta_{\rho}$ is the added arrow in $Q_{T(A')}$ with respect to the path $\rho$, see Section \ref{sec-trivial-extensions}. Following Lemma \ref{lem: paths Asg-A'}, the induced paths ${}^{\varepsilon} \rho ^{\varepsilon'}$, for all possible   $\varepsilon,\varepsilon'\in \{+,-,\emptyset\}$, are maximal paths in $A^{\text{sg}}$. By Lemma \ref{lemma:basis of soc A}  we can give a basis $\mathcal{M}$ of the socle of $A^{\text{sg}}$ such that ${}^{\varepsilon} \rho ^{\varepsilon'}\in \mathcal{M}$. 
Then, by  Definition \ref{def:elementary}  there exist arrows ${}^{\varepsilon'} \beta ^{\varepsilon}$ in $Q_{T(A^{\text{sg}})}$ such that $({}^{\varepsilon} \rho ^{\varepsilon'})({}^{\varepsilon'} \beta ^{\varepsilon})$ are elementary cycles in $T(A^{\text{sg}}).$ 
Notice that the arrows ${}^{\varepsilon'} \beta ^{\varepsilon}$ in $Q^{\text{sg}}_{T(A')}$, for all possible $\varepsilon,\varepsilon'\in \{+,-,\emptyset\}$, are induced by the arrow $\beta_{\rho}$ in $T(A')$. Therefore, the cycles ${}^{\varepsilon} C ^{\varepsilon}=({}^{\varepsilon} \rho ^{\varepsilon'})({}^{\varepsilon'} \beta ^{\varepsilon})$ in $T(A^{\text{sg}})$ are elementary cycles in $T(A^{\text{sg}})$ induced by $C$.

With similar arguments, one can show that any elementary cycle ${}^{\varepsilon} C ^{\varepsilon}$ in $T(A^{\text{sg}})$ is induced by an elementary cycle $C$ in $T(A')$.\end{proof}

\begin{thm}\label{teo:ideal relaciones}
 Let $A=\SG$ be a skew-gentle algebra, let $A^{\mathrm{sg}}$ be its admissible presentation, and $A'=KQ'/I'$ be its auxiliary gentle algebra. Then the trivial extension $T(A^{\mathrm{sg}})$ is isomorphic to the sg-bound quiver algebra of the trivial extension $T(A')$, with $\Qsp$ as the set of special vertices and the set of elementary cycles $\mathcal{C}$ of $T(A')$ as the set of special cycles.
\end{thm}

\begin{proof}
Let $A=\SG$ be a skew-gentle algebra and let $A^{\text{sg}}=KQ^{\text{sg}}/I^{\text{sg}}$ and $A'=KQ'/I'$ be the algebras as described in the statement. By Lemma~\ref{lemma:quiver of trivial extension}, the quivers $Q_{T(A^{\text{sg}})}$ and $Q_{T(A')}$ of the trivial extensions $T(A^{\text{sg}})$ and $T(A')$ coincide, then it is enough to verify that the ideal $I_{T(A^{\text{sg}})}$ of the trivial extension coincides with the sg-ideal $I^{\text{sg}}_{T(A')}$ of $T(A')$.

$\bullet$ Let $x$ be a vertex in $\Qsp$, and let $\alpha, \beta$ be arrows in $Q_{T({A'})}$ such that $\cdot \xrightarrow{\alpha}x\xrightarrow{\beta}\cdot$. The commutative relations $({}^\varepsilon \alpha ^{+})({}^+\beta ^{\varepsilon '})- ({}^\varepsilon \alpha ^{-})({}^-\beta ^{\varepsilon '})$ in  $I^{\text{sg}}_{T(A')}$ are of type (a). For the ideal $I_{T(A^{\text{sg}})}$ we have to consider two cases:

\begin{enumerate}

\item If $\alpha$ and $\beta$ are arrows of $Q'$, then 
such commutative relations
are generators of $I_{A^{\text{sg}}}$, and therefore also of $I_{T(A^{\text{sg}})}$.

 \item We assume that $\beta$ (or $\alpha$) is not an arrow of $Q'$. Then, the arrows ${}^{\varepsilon''}\beta^{\varepsilon'}$ in $Q_{T(A^{\text{sg}})}$ do not belong to $Q_{A^{\text{sg}}}$, with $\varepsilon''\in \{+,-\}$, and $\varepsilon'\in \{+,-,\emptyset\}$.
That it to say that the mentioned arrows are new arrows in the quiver $Q_{T(A^{\text{sg}})}$ which are in correspondence with  maximal paths $({}^{\varepsilon'}\rho^{\varepsilon})(^{\varepsilon}\alpha^{\varepsilon''})$ in $A^{\text{sg}}$, with  $\varepsilon\in \{+,-,\emptyset\}$. Therefore, the commutative relations $({}^\varepsilon \alpha ^{+})({}^+\beta ^{\varepsilon '})- ({}^\varepsilon \alpha ^{-})({}^-\beta ^{\varepsilon '})$ in  $I_{T(A^{\text{sg}})}$ arise from the third condition of \cite[Theorem A.1]{FSTTV22}, since the paths $({}^\varepsilon \alpha ^{+})({}^+\beta ^{\varepsilon '})$  and $({}^\varepsilon \alpha ^{-})({}^-\beta ^{\varepsilon '})$ have the same complement 
${}^{\varepsilon'}\rho^{\varepsilon}$ in an elementary cycle. 
\end{enumerate}

$\bullet$ Let $C_1$ and $C_2$ be two elementary cycles of $T(A')$  starting at the same 
non-special vertex. Then, by Lemma \ref{lem: elementary cycles}, we have that ${}^{\varepsilon}C_1^{\varepsilon}$ and ${}^{\varepsilon}C_2^{\varepsilon}$ are elementary cycles of $Q(A^{\text{sg}})$, with $\varepsilon=\emptyset$. 
Therefore, it follows from the third condition of \cite[Theorem A.1]{FSTTV22} that ${}^{\varepsilon}C_1^{\varepsilon}-{}^{\varepsilon}C_2^{\varepsilon}$ 
are generators of $I_{T(A^{\text{sg}})}$, and for the sg-ideal $I^{\text{sg}}_{T(A')}$ of $T(A')$, such relations arise from relations of Type (b) in Definition \ref{definition-admissible quiver}.

$\bullet$ Finally, monomial relations of $I_{T(A^{\text{sg}})}$ arise from monomial relations of $I^{\text{sg}}$ or from paths which  are not contained in any elementary cycle of $T(A^{\text{sg}})$. On the other hand,  monomial relations of the sg-ideal $I^{\text{sg}}_T(A')$ of $T(A')$ arise from relations of Type (c) and (d) of Definition \ref{definition-admissible quiver}. It follows from Lemma \ref{lem: elementary cycles} that these two mentioned sets are equal.\end{proof}

In the following example, we illustrate how to use Theorem \ref{teo:ideal relaciones} to compute the trivial extension of a skew-gentle algebra starting from its auxiliary gentle algebra.

\begin{exmp}\label{toy-example2}
Let $A$ be the skew-gentle algebra given in Example \ref{ex: toy example} and let $A'$ be its auxiliary gentle algebra. 

Following \cite[Lemma 2.10]{assem2006}, the ordinary quiver $Q_{T(A')}$ of $T(A')$ is depicted in Figure~\ref{fig:ordinary quiver}. Here $\beta_{p_1}$ and $\beta_{p_2}$ correspond to the new arrows in $Q_{T(A')}$ associated with the maximal paths $p_1=\alpha\beta\gamma$ and $p_2=\delta\lambda$ in $A'$, respectively. As a consequence, the set of elementary cycle (up to cycle equivalences) in $T(A')$ is given by $C_1=\alpha\beta\gamma\beta_{p_1}$ and $C_2=\delta\lambda\beta_{p_2}$.

\begin{figure}
    \centering
    \begin{tikzcd}
 &  &  & 5\arrow[ld,"\lambda", swap] &  \\
1 \arrow[r,swap,"\alpha"]  & 2 \arrow[r,swap, "\beta"] & 3\arrow[rr, shift left, "\beta_{p_2}"] \arrow[rr,swap, shift right, "\gamma"] &  & 4 \arrow[llll, bend left,"\beta_{p_1}"] \arrow[lu, "\delta",swap]& 
\end{tikzcd}
    \caption{Ordinary quiver of the trivial extension $T(A')$ of the auxiliary gentle algebra $A'$}
    \label{fig:ordinary quiver}
\end{figure}

Then by Theorem~\ref{ideal TA}, the ideal $I_{T(A')}$ is generated by
\begin{enumerate}
    \item[(1)] The relations in $I_{A}$: $\gamma\delta, \lambda\gamma$;
    \item[(2)] The paths that are not contained in any elementary cycle, which in this case are of two kinds.
    
    \begin{enumerate}
        \item The composition of any two consecutive arrows in different elementary cycles: $\beta\beta_{p_2}$,$\beta_{p_2}\beta_{p_1}$.
        \item The paths of the form $C\eta$, where $C$ is an elementary cycle and $\eta$ is the first arrow of $C$: $\alpha\beta\gamma\beta_{p_1}\alpha$, $\beta\gamma\beta_{p_1}\alpha\beta$, $\gamma\beta_{p_1}\alpha\beta\gamma$, $\beta_{p_1}\alpha\beta\gamma\beta_{p_1}$, $\delta\lambda\beta_{p_2}\delta$, $\lambda\beta_{p_2}\delta\lambda$, $\beta_{p_2}\delta\lambda\beta$.
    \end{enumerate}
    \item[(3)] Linear combinations of paths satisfying (3) in Theorem~\ref{ideal TA}: $\beta_{p_2}\delta\lambda-\gamma\beta_{p_1}\alpha\beta, \delta\lambda\beta_{p_2}-\beta_{p_1}\alpha\beta\gamma$.
\end{enumerate}

It follows from Theorem \ref{teo:ideal relaciones}, that the trivial extension $T(A^{\text{sg}})$ of $A^{\text{sg}}$ is isomorphic to the sg-bound quiver algebra of $T(A')$ with $\Qsp=\{1,2\}$ as the set of special vertices and $\mathcal{C}=\{\alpha\beta\gamma\beta_{p_1}, \delta\lambda\beta_{p_2}\}$ as the set of special cycles. Then $T(A^{\text{sg}})$ is the algebra with quiver

\begin{figure}[h!]
\begin{tikzcd}
1^+ \arrow[rr, "^{+}\alpha^+"] \arrow[rrd, "^{+}\alpha^-" near start]   &  & 2^+ \arrow[rd, "^{+}\beta"] &                        & 5 \arrow[ld, "\lambda"'] &                         \\
1^- \arrow[rr, "^{-}\alpha^-"'] \arrow[rru, "^{-}\alpha^+"' near start] &  & 2^- \arrow[r, "^{-}\beta"]  & 3 \arrow[rr, shift left, "\beta_{p_2}"] \arrow[rr,swap, shift right, "\gamma"] &                          & 4 \arrow[lu, "\delta"'] \arrow[lllll, bend left, "\beta_{p_1}^{-}"] \arrow[lllllu, bend right=60, "\beta_{p_1}^{+}", swap]
\end{tikzcd}
\caption{Quiver of the trivial extension $T(A^{\text{sg}})$}
\label{Fig: trivial extension 1}
\end{figure}
and the ideal $I_{T(A^{\text{sg}})}$ is generated by:
\begin{itemize}
    \item[Type a: ] $({}^{+}\alpha^{+})({}^{+}\beta)- ({}^{+}\alpha^{-})({}^{-}\beta)$, $({}^{-}\alpha^{+})({}^{+}\beta)- ({}^{-}\alpha^{-})({}^{-}\beta)$,
    \item[Type b: ] $\beta_{p_2}\delta\lambda-\gamma(\beta_{p_1}^-)({}^-\alpha^-)({}^-\beta)$, $ \beta_{p_2}\delta\lambda-\gamma(\beta_{p_1}^+)({}^+\alpha^+)({}^+\beta)$, $\delta\lambda\beta_{p_2}-(\beta_{p_1}^+)({}^+\alpha^+)({}^+\beta)\gamma$ and \newline$\delta\lambda\beta_{p_2}-(\beta_{p_1}^-)({}^-\alpha^-)({}^-\beta)\gamma$,
    \item[Type c: ]  $({}^+\beta)\beta_{p_2}, ({}^-\beta)\beta_{p_2},\beta_{p_2}(\beta_{p_1}^+), \beta_{p_2}(\beta_{p_1}^-)$, $({}^+\alpha^+)({}^+\beta)\gamma(\beta_{p_1}^+)({}^+\alpha^+)$, $({}^+\beta)\gamma(\beta_{p_1}^+)({}^+\alpha^+)({}^+\beta)$, \newline $\gamma(\beta_{p_1}^+)({}^+\alpha^+)({}^+\beta)\gamma$, $(\beta_{p_1}^+)({}^+\alpha^+)({}^+\beta)\gamma(\beta_{p_1}^+)$, $\delta\lambda\beta_{p_2}\delta, \lambda\beta_{p_2}\delta\lambda, \beta_{p_2}\delta\lambda\beta$,
    
    \item[Type d: ] $({}^+\alpha^+)({}^+\beta)\gamma(\beta_{p_1}^+)({}^+\alpha^-), ({}^-\alpha^-)({}^-\beta)\gamma(\beta_{p_1}^-)({}^-\alpha^+)$.
\end{itemize}

\end{exmp}

\section{Skew-Brauer graph algebras and trivial extension of skew-gentle algebras}

In this section  we define a generalization of Brauer graph algebras, that we call skew-Brauer graph algebras, and we show that the trivial extension of any skew-gentle algebra is isomorphic to a skew-Brauer graph algebra.

Given an undirected graph $\Gamma$, we denote by $\val(v)$
the \emph{valency} of a vertex $v$, i.e. the amount of edges incident to $v$ in $\Gamma$.

We recall that a \emph{Brauer graph} $\Gamma$ is a tuple $(\Gamma_0, \Gamma_1, \mf, \of)$  where $(\Gamma_0, \Gamma_1)$ is an unoriented graph with  $\Gamma_0$ the set of vertices and  $\Gamma_1$ the set of edges and

\begin{enumerate}
    \item $\mf$ is a \emph{multiplicity function} $\mf : \Gamma_0 \to \mathbb{Z}^+$, assigning a positive number to each $v \in \Gamma_0$.
    \item $\of$ is an \emph{orientation} of the edges attached to each $v \in \Gamma_0$. For a vertex $v$ such that $\val(v) = t$ we denote $\of (v) = (x_1, \ldots, x_t)$. In the particular case that $\val (v) = 1$ and the only edge attached to $v$ is $x$, if $\mf(v) =1$ the cyclic ordering is given by the single element $x$, but when $\mf(v) > 1$ we set $\of(v) = (x,x)$
    so that in the cyclic order the successor (and the predecessor) of $x$ is $x$. 
\end{enumerate}

\begin{defn}\label{def: skew-brauer graph}
 A \emph{skew-Brauer graph} $\Gamma^{\times}$ is tuple $\Gamma^{\times}=(\Gamma'_0,\Gamma_0^{\times}, \Gamma_1, \mf,\of)$ such that $\Gamma=(\Gamma_0'\sqcup \Gamma_0^{\times}, \Gamma_1, \mf,\of)$ is a Brauer graph with $\Gamma'_0\sqcup\Gamma_0^{\times}$ a disjoint union of the set of vertices of the graph $\Gamma$, and such that any vertex $v \in \Gamma^\times_0$ satisfies $\mf(v)=1$, $\of(v)=(x_v)$, and  $v$ is adjacent to a single vertex $w\in \Gamma'_0$.

 We call the vertices of $\Gamma_0^{\times}$ \emph{special vertices} and the edges attached to a special vertex are called \emph{special edges}. To emphasize that a vertex $v$ in $\Gamma$ is special, we replace $v$ by a decorated vertex $v^{\times}$.

Given a skew-Brauer graph $\Gamma^{\times}$, we call the underlying Brauer graph $\Gamma$ the \emph{auxiliary Brauer graph}, and we call $B_{\Gamma}$ the \emph{auxiliary Brauer graph algebra}.
 \end{defn}

\begin{exmp}\label{example-1} Consider the graph in Figure \ref{fig:example-1} where $\mf(v_i^{\times})=1$ for $i\in \{0,1\}$, $\mf(v_4)=1$ and $\mf(v_j) \in \mathbf{N}$ for $j\in \{2,3\}$.

\begin{figure}[h!]
   \adjustbox{scale=.9,center}
{\begin{tikzcd}
    v_1^{\times} \arrow[rd, dash, "2"] \\
     & v_2 \arrow[rr, bend left, dash, "3"] \arrow[rr, bend right,dash, swap,"4"] &   v_4  & v_3 \arrow[l,dash, "5"]\\
    v_0^\times \arrow[ru, dash,swap, "1"] & & 
    \end{tikzcd}}
    \caption{Skew-Brauer graph with two special vertices.}
    \label{fig:example-1}
\end{figure}
This is a skew-Brauer graph with two special vertices $v_1^{\times}$ and $v_0^{\times}$. We have $\of(v_0^{\times}) = (1)$, $\of(v_1^{\times}) = (2)$, $\of(v_2) = (1,2,3,4) $, $\of(v_3) = (3,4,5)$ and $\of (v_4) =(5)$.
\end{exmp}

Following \cite{Sch18}, given a Brauer graph $\Gamma$, we define the Brauer quiver $Q_{\Gamma}=(Q_0, Q_1)$ as follows: 

\begin{enumerate}
\item The set of vertices $Q_0$ is in bijection with the edges of $\Gamma$. 
\item There is an arrow $x \to y$ each time $x$ and $y$ are edges incident to a vertex $v \in \Gamma_0$ and $y$ is the immediate successor of $x$ in the cyclic order $\of(v)$.
\item Each non-distinguished vertex $v$ of $\Gamma$ with $\mf(v)\mathrm{val}(v)\geq 2$ gives rise an oriented cycle $C_v \colon  x_1 \xrightarrow{\alpha_1} x_2 \to \cdots x_{t} \xrightarrow{\alpha_{t}} x_1$ in $Q$, up to cycle equivalence, established by the orientation $\of$, i. e. if $\of(v) = (x_1, \ldots, x_t)$ (that we will set as clockwise in our examples when $t>1$). 
We say that $C_v$ is \emph{a special cycle at $v$}.
\end{enumerate}

We define the skew-Brauer quiver $Q_{\Gamma^{\times}}$ from a graph $\Gamma^{\times}$ as a pair $Q_{\Gamma^{\times}}=(Q_{\Gamma}, \Qsp)$ where $Q_{\Gamma}$ is the Brauer quiver of $\Gamma$ and $\Qsp$ is the set of special vertices that is in bijection with the set $\Gamma_0^{\times}$. As before, to emphasize that a vertex $y$ is special in $Q_{\Gamma^{\times}}$, we replace $y$ by a decorated vertex $y^{\times}$.

Finally, given a skew-Brauer quiver $Q_{\Gamma^\times}$, we define its \emph{admissible skew-Brauer quiver} $Q^{\mathrm{sg}}_{\Gamma^\times}$ as the one constructed from $Q_{\Gamma^{\times}}$ following the  rules described in Definition~\ref{definition-admissible quiver}.

Following the previous definitions, the skew-Brauer graph in Figure~\ref{fig:example-1} defines the skew-Brauer quiver $Q_{\Gamma^{\times}}$ depicted in Figure~\ref{fig:example-1b}, and its admissible skew-Brauer quiver coincides with the quiver given in Figure~\ref{Fig: trivial extension 1}.

\begin{figure}[h!]
\adjustbox{scale=.9,center}
{\begin{tikzcd}
&  &  & 5\arrow[ld,"\lambda", swap] &  \\
1^{\times} \arrow[r,swap,"\alpha"]  & 2^{\times} \arrow[r,swap, "\beta"] & 3\arrow[rr, shift left, "\beta_{p_2}"] \arrow[rr,swap, shift right, "\gamma"] &  & 4 \arrow[llll, bend left,"\beta_{p_1}"] \arrow[lu, "\delta",swap]& 
\end{tikzcd}}
\caption{Skew-Brauer quiver with two special vertices}
\label{fig:example-1b}
\end{figure}

Note that given a skew-Brauer graph $\Gamma^\times$, there is exactly one special cycle $C_v$ in $Q_{\Gamma^{\times}}$ for every $v\in \Gamma_0$ with $\mf(v)\mathrm{val}(v)\geq 2$, up to cycle permutation, and each special cycle $C_v$ induces $2^r$ cycles $\{C_{v,1}, C_{v,2}, \dots, C_{v,2^r}\}$ in  {{$Q^{\mathrm{sg}}_{\Gamma^\times}$}} , where $r$ is the number of special edges attached to $v$, which we call sg-\emph{special cycles at $v$.} Each one of the cycles $C_{v,l}$ and $C_v$ are considered up to cyclic permutation.

\subsection{Skew-Brauer algebras and a set of relations}\label{subsec-skew-brauer algebra} 
Let $\Gamma^{\times}$ be a skew-Brauer  {{graph}}, where $Q_{\Gamma^{\times}}$ is its skew-Brauer quiver with special set of vertices $\Qsp$, and let $Q^{\mathrm{sg}}_{\Gamma^{\times}}$ be its admissible skew-Brauer quiver. We define the ideal $I_{\Gamma^\times}$ in $KQ^{\mathrm{sg}}_{\Gamma^\times}$ as the one generated by the following four types of relations:

\emph{Type 0:} Difference of two paths of the form
$$({}^\varepsilon\alpha^+) ({}^{+} \beta^{\varepsilon'})-({}^{\varepsilon}\alpha^-)({}^{-}\beta^{\varepsilon'})$$
for all possible  $\varepsilon=\{+,-, \emptyset\}$ and $\varepsilon' = \{+,-, \emptyset\}$.

\smallskip

\emph{Type I:} Difference of powers of two sg-special cycles $C_{v, l}$ and $C_{v', l'}$  at vertices $v$ and $v'$ of $\Gamma$
$$C_{v, l}^{\mf(v)}-C_{v', l'}^{\mf(v')}$$

where both cycles start at a vertex $y\in  Q^{\mathrm{sg}}_{\Gamma^\times}$ with $\mf(v) \val(v) \neq 1$ and $\mf(v') \val(v') \neq 1$.

\smallskip

\emph{Type IIa:} Paths of the form
$$C_{v,l}^{\mf(v)}\alpha$$ 

where $\alpha$ is the first arrow of the sg-special cycle $C_{v,l}$.

\emph{Type IIb:}  Paths of the form 
$$C_{v,l}^{\mf(v)-1}({}^{\varepsilon_1}\alpha_1^{\varepsilon_2}) \cdots ({}^{\varepsilon_r}\alpha_r^{\varepsilon_{r+1}})$$

for each special cycle $C_{v}=\alpha_1 \cdots \alpha_r$ in $Q_{\Gamma^{\times}}$ such that $s(\alpha_1)$ is a special vertex of $Q_{\Gamma^{\times}}$, for all possibles  $\varepsilon_k=\{+,-, \emptyset\}$ where $k=1, \dots, r+1$ with $\varepsilon_1\neq \varepsilon_{r+1}$.

\smallskip

\emph{Type III:} Paths of length two of the form
$$\alpha\beta$$

where $\alpha\beta$ is not a sub-path of any sg-special cycle $C_{v, l}$, except when $\alpha=\beta$ and $\alpha$ is a loop induced by a vertex $v$ in $\Gamma^{\times}$, i.e. where $\val(v) = 1$ and $\mf(v)>1$. \footnote{See Examples 2.2 and 2.3 (1) in \cite{Sch18}.}


\begin{defn}\label{def: skew brauer algebra}
Let $\Gamma^{\times}=(\Gamma_0', \Gamma_0^{\times}, \Gamma_1,  \mf, \of)$ be a skew-Brauer graph. The algebra $B_{\Gamma^{\times}}=KQ^{\mathrm{sg}}_{\Gamma^\times}/ I_{\Gamma^{\times}}$ is called the \emph{skew-Brauer}  graph algebra associated to the skew-Brauer graph $\Gamma^\times$.
\end{defn}

\begin{rem}\label{rem:brauer algebras subconjunto de skew brauer algebras}
Note that any skew-Brauer graph {{$\Gamma^{\times}=(\Gamma'_0,\Gamma_0^{\times}, \Gamma_1, \mf, \of)$}}
with $\Gamma_0^\times=\emptyset$ is a Brauer graph. Therefore, in that case, the definition of skew-Brauer graph algebra coincides with the definition of Brauer graph algebra given in \cite{Sch15}.
\end{rem}

\begin{thm}\label{thm: skew brauer simetricas} Skew-Brauer graph algebras (with arbitrary multiplicity function $\mf$) are finite dimensional and symmetric.
\end{thm}

\begin{proof}
To show that $B_{\Gamma^{\times}}$ is a finite dimensional algebra, it is enough to show that the ideal $I_{\Gamma^{\times}}$ is admissible. Since any sg-special cycle $C_{v,l}$ is either a cycle of length at least 2 or a loop surrounding a $v$ with multiplicity function $\mf(v)>1$, we have that $I_{\Gamma^\times}\subset J^2$, where  $J^2$ denotes the ideal generated by paths of length 2.
 
Moreover, it follows from the relations of type I, IIa, IIb, and III, that any path of length at least  $N+1$, where \[N=\text{max} \, \{ \text{length} \, C_{v, l}^{\mf(v)}\},\] is an element of the ideal $I_{\Gamma^{\times}}$, therefore $I_{\Gamma^\times}$ is admissible, as we claimed. 

By \cite{Yam96}, to prove that $B_{\Gamma^\times}$ is symmetric it is enough to prove that the $K$-lineal form $\phi: B_{\Gamma^\times} \to K$ given by $$\phi(x)=\begin{cases} 1 & \text{if $x=C_{v,l}^{\mf(v)}$ is the $\mf(v)$-power of a sg-special cycle for some $v\in \Gamma^{\times}$}\\
0 & \text{otherwise}
\end{cases}$$

satisfies $\phi(ab)=\phi(ba)$ for any paths $a,b$ in $B_{\Gamma^\times}$.

Let $a$ and $b$ be any two paths in $B_{\Gamma^\times}$ such that $ab$ is not zero. By relations of type IIa, IIb, and III, we notice that $a$, $b$ and $ab$ is a sub-path of the $\mf(v)$-power of a sg-special cycle $C_{v,l}^{\mf(v)}$ for some vertex $v\in\Gamma^{\times}$. If $ab=C_{v,l}^{\mf(v)}$, then $t(ab)=t(b)=s(a)$, and as a consequence $ba$ is a cyclic permutation of  $C_{v,l}^{\mf(v)}$, then $\phi(ba)=1=\phi(ab)$. It is easy to see that if  $ab$ is not the $\mf(v)$-th power of any sg-special cycle $C_{v,l}$, then $\phi(ab)=0=\phi(ba)$.
\end{proof}

\begin{rem}\label{rem: projectives}
The following example shows a simple way to compute the indecomposable projective modules of skew-Brauer graph algebras corresponding to vertices of $Q^{\mathrm{sg}}_{\Gamma^\times}$ using its associated Brauer graph. It also shows that skew-Brauer graph algebras are not special biserial.
\end{rem}

\begin{exmp}\label{ex:projectives}
 Let $\Gamma^\times=(\Gamma'_0, \Gamma^\times_0, \Gamma_1, \mf, \of)$ be the skew-Brauer graph  of the Example \ref{example-1}, and $\Gamma=(\Gamma'_0 \sqcup\Gamma^\times_0, \Gamma_1, \mf, \of)$ its associated Brauer graph. For simplicity we set $\mf(v_i)=1$ for every $i=0,1,2,3,4,5$. Following \cite{Sch18}, one can compute the indecomposable projective $B_\Gamma$-module $P_i$ of each vertex $i\in Q_{\Gamma}$ from $\Gamma$, for example, the indecomposable modules $P_2$ and $P_3$ (denoted by their composition factors) are as follows

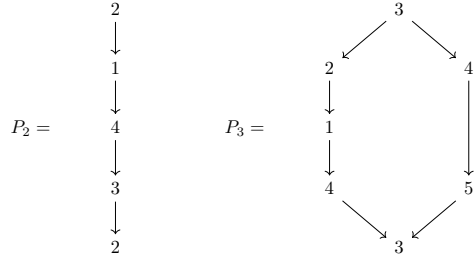
\begin{figure}[H]
\adjustbox{scale=.7,center}
{
\begin{tikzcd}
 & 2\arrow[d, no head] & &  & & 3 \arrow[dl, no head]\arrow[dr, no head]& \\
& 3 \arrow[d, no head] & & & 4\arrow[d, no head]& & 4\arrow[dd, no head]\\
P_2= &4 \arrow[d, no head]& & P_3= &1\arrow[d, no head] & & \\
& 1\arrow[d, no head] &  & &2\arrow[dr, no head] & &5\arrow[dl, no head]\\
& 2 & & & &3 &
\end{tikzcd}}
\end{figure}

Since the vertex $2$ corresponds to a special vertex in $Q_{\Gamma^{\times}}$, then there are two indecomposable projectives $B_{\Gamma^\times}$-modules ${P_2^+}$ and ${P_2^-}$, both of them are obtained from $P_2$ by replacing every special vertex $i$ in their composition series of $rad(P_2)/soc(P_2)$ by $i^+$ and $i^-$ and adding the respective linear maps induced by the arrows in $Q^{\mathrm{sg}}_{\Gamma^\times}$, and then replacing the socle and the top by $2^+$ and $2^-$ respectively. While, for $P_3$ there is only one projective $B_{\Gamma^\times}$-module $\widehat{P_3}$ which is obtained from $P_3$ by replacing every special vertex $i$ in their composition series of $rad(P_2)/soc(P_2)$ by $i^+$ and $i^-$ and considering their respective linear maps induced by the arrows in $Q^{\mathrm{sg}}_{\Gamma^\times}$. Then the projectives $P_2^+, P_2^-$ and $\widehat{P_3}$ are as follows.

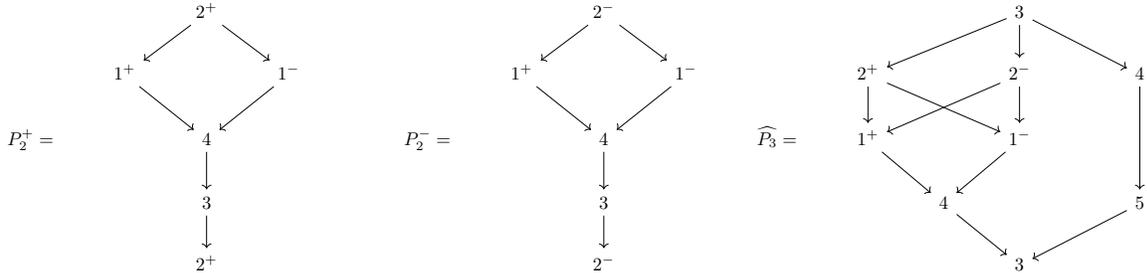
\begin{figure}[h!]
\adjustbox{scale=.7,center}
{\begin{tikzcd}
 & & 2^+ \arrow[d, no head] &&&& 2^-\arrow[d, no head] &&&&& 3 \arrow[ld, no head] \arrow[rd, no head] & \\ 
 && 3 \arrow[d, no head] &&&& 3 \arrow[d, no head] &&&& 4 \arrow[d, no head] \arrow[ld, no head]   && 4 \arrow[dd, no head] \\
P_{2^+} = && 4 \arrow[ld, no head] \arrow[rd, no head] && P_{2^-}= && 4 \arrow[ld, no head] \arrow[rd, no head] && \widehat{P_3}= & 1^+ \arrow[d, no head] \arrow[rd, no head] & 1^- \arrow[d, no head] \arrow[ld, no head] && \\
& 1^+ \arrow[rd, no head] && 1^- \arrow[ld, no head] && 1^+ \arrow[rd, no head] && 
1^- \arrow[ld, no head] && 2^+ \arrow[rrd, no head] & 2^- \arrow[rd, no head] && 5 \arrow[ld, no head] \\
&& 2^+  &&&& 2^- &&&&& 3 & 
\end{tikzcd}
}
\end{figure}
\end{exmp}

\subsection{The trivial extension of a skew-gentle algebra}\label{subsec the trivial extension}

In this section, we state a relation between trivial extension of skew-gentle algebras and a subfamily of skew-Brauer graph algebras. To do that, we adapt the construction of generalized ribbon graphs given in \cite[Remark 3.4]{LSV}. In particular, we recall the definition of $sp$-maximal paths, which generalizes the notion of maximal paths in algebras defined by non-admissible quiver presentations.

Given a skew-gentle algebra $A=\SG$, we say that a path $p\neq 0$ in $A$ is $sp$-maximal if $p\alpha=\alpha p=0$ in $A$ for every arrow $\alpha$ in $Q$ that is not a special loop, and if $s(p)$ (resp. $t(p)$) is a special vertex then the first arrow (resp. the last arrow) of $p$ is the special loop attached to it. 

Notice that, given a skew-gentle algebra $A=\SG$, $f=f^k$ modulo $I$ for any special loop $f$ and for any $k\in\mathbb N$. Therefore, we can assume without loss of generality that if a path $p$ contains a special loop, the exponent of $f$ is one.

\begin{exmp}
Let $A_1=A_1(Q,I_1, \textrm{Sp}_1)$ be a skew-gentle algebra with quiver $Q$
\[\begin{tikzcd}
    1 \arrow[r,swap,"\alpha"]  & 2 \arrow[r,swap, "\beta"]\arrow[out=60,in=120,loop  , distance=1.6em, swap, "f_1"] & 3 \arrow[r, swap, "\gamma"]& 4 \arrow[out=60,in=120,loop  , distance=1.6em,swap, "f_2"]
\end{tikzcd}\]
ideal $I_1=\langle\alpha\beta, \beta\gamma, f_1^2\rangle$ and $\textrm{Sp}_1=\{f_2\}$ the set of special loops. Then the set of $sp$-maximal paths in $A_1$ is given by $\alpha f_1 \beta$ and $\gamma f_2$. For the skew-gentle algebra $A_2(Q,I_2,\textrm{Sp}_2)$ sharing the same quiver as $A_1$, but with the ideal $I_2=\langle\alpha\beta, \beta\gamma\rangle$, and the set of special vertices $\textrm{Sp}_2=\{f_1, f_2\}$, the set of $sp$-maximal paths in $A_2$ coincides with the set for $A_1$.
\end{exmp}

We define a skew-Brauer graph $\Gamma^\times_{A}=\Gamma^\times$ associated with a skew-gentle algebra $A=\SG$ as follows. 

\begin{itemize}
    \item The set of vertices $\Gamma_0'$ of  $\Gamma_A^\times$ is in bijection with the union of 
    
    \begin{itemize}
        \item[i)] all sp-maximal paths in $A$, 
        \item[ii)] trivial paths $e_x$ such that $x$ is either the source or the target of only one arrow or $x$ is the target of exactly one arrow $\alpha$ and the source of exactly one arrow $\beta$, and $\alpha \beta\neq 0$  in $Q$.
    \end{itemize}
    
    \item The set of special vertices $\Gamma^\times_0$  is in bijection with trivial paths $e_x$ such that $x$ is a special vertex of $Q$.
    \item The set of edges of $\Gamma^\times_A$ is in bijection with the vertices of $Q_0$ (note that this includes the special vertices).
\end{itemize}

The multiplicity function $\mf$ is the constant function equal to 1, and the cyclic ordering of the edges attached to each vertex $v_0\in \Gamma_0$ is given either by the order of the sp-maximal path associated to $v_0$ or $\of (v)=(e_x)$ otherwise.

\begin{lem}\label{lem:maximal-paths AyA'}
Let $A=\SG$ be a skew-gentle algebra and $A'$ be its auxiliary gentle algebra.
Then the sp-maximal paths of the skew-gentle algebra $A$ are in bijection with the maximal paths of the auxiliary gentle algebra $A'$.
\end{lem}

\begin{proof}
Let $\rho=\alpha_1\dots \alpha_k$ be a $sp$-maximal path in $A$ and $\rho'=\beta_1\dots \beta_k$ be the path in $A'$ where $\beta_i=\alpha_i$ if $\alpha_i$ is not a special loop and $\beta_i=e_{t(\alpha_{i-1})}$ otherwise. We claim that $\rho'$ is a maximal path in $A'$.

Notice that $\rho'\neq 0$, otherwise $\beta_i\beta_{i+1}$ is a zero relation in $A'$ for some $i$, which implies that $\beta_i=\alpha_i$ and $\beta_{i+1}=\alpha_{i+1}$ and $\alpha_i\alpha_{i+1}$ is a zero relation in $A$, as a consequence $\rho$ is not a sp-maximal path.

Suppose $\rho'$ is not a maximal path in $A'$. Then there exists at least one arrow in $Q_{A'}$ such that $\alpha\rho'\neq 0$ or $\rho'\alpha\neq 0$ in $A'$. Without loss of generality, suppose that $\alpha\rho'\neq 0$ in $A'$. We notice that if the start $s(\rho)=x$ of $\rho$ is not a special vertex, then $\alpha$ is not a special loop and $\alpha$ is an arrow of $Q_{A}$, then $\alpha\rho$ is not zero in $A$, a contradiction with the hypothesis. Now, if $x$ is a special vertex, then the first arrow $\alpha_1$ of $\rho$ is the special loop attached to $x$ and $\alpha\rho$ is not zero in $A$, again a contradiction with the hypothesis.
\end{proof}

\begin{thm}\label{thm: equiv1}
Let $A=\SG$ be a skew-gentle algebra and $\Gamma_{A}^{\times}=(\Gamma_0', \Gamma_0^{\times}, \Gamma_1,  \mf= 1, \of)$ be its skew-Brauer graph. Then $T(A^{\mathrm{sg}})\simeq B_{\Gamma_{A}^{\times}}$.
\end{thm}

\begin{proof}
Let $A'$ be the auxiliary gentle algebra of $A$. As a consequence of Lemma~\ref{lem:maximal-paths AyA'}, it is easy to check that $\Gamma_A^{\times}$ as Brauer graph, namely, the Brauer graph $\Gamma=(\Gamma_0'\cup \Gamma^\times_0,\Gamma_1,\mf, \of)$ coincides with the Brauer graph $\Gamma_{A'}$.

By \cite[Theorem 1.2]{Sch15}, the trivial extension  $T(A')$ of $A'$ is isomorphic to the Brauer graph algebra $B_{\Gamma_{A'}}$
and moreover, by Theorem \ref{teo:ideal relaciones}  we have that $T(A)$ is isomorphic to $KQ^{\mathrm{sg}}_{T(A')}/I^{\mathrm{sg}}_{T(A')}$. It is easy to check that the last presentation is isomorphic to the skew-Brauer graph algebra $B_{\Gamma_{A}^{\times}}$ described in Section \ref{subsec-skew-brauer algebra}. Therefore, the statement follows.
\end{proof}

To prove the second theorem of this section, we need to recall the definition of \emph{admissible cut} from \cite{FP02}. Let $T(A)$ be the trivial extension of a bound quiver algebra $A$ and let $\mathcal{D}$ be a set of arrows of $Q_{T(A)}$. We say that $\mathcal{D}$
is an \textit{admissible cut} of $Q_{T(A)}$ if it consists of exactly one
arrow of each elementary cycle in $Q_{T(A)}$ (see Definition~\ref{def:elementary}). Moreover, we say that the algebra $B$ is an \textit{admissible cut} of $T(A)$ if it is the
quotient algebra $T(A)/\langle \mathcal{D}\rangle$, where $\mathcal{D}$ is an admissible cut of $Q_{T(A)}$. We say that $T(A)/\langle \mathcal{D}\rangle$ is the \emph{quotient algebra of $T(A)$ induced by $\mathcal{D}$.}

Notice that the definition of admissible cut is given for trivial extensions of bound path algebras. However, it is possible to generalize this idea to give a recognition theorem of trivial extension, see \cite{FSTTV22} for details. To do this, given an algebra $A$ we consider a finite set $\mathcal{C}$ of non-zero cycles of length at least two, then a \emph{cut set} $\mathcal{D}$ of $A$ with respect to $\mathcal{C}$ is
a subset of arrows $\{\alpha_1, \dots, \alpha_k\}$ of $Q_A$ containing exactly one arrow in each cycle of $\mathcal{C}$, and such that $\alpha_i$ appears at most once in each cycle $C$ of $\mathcal{C}$. This definition corresponds to definition of \emph{allowable cut} given in \cite{FSTTV22}.

\begin{rem} If $A$ is a trivial extension of a bound path algebra $B$, then any admissible cut is indeed a cut set with respect to the set of the elementary cycles.
\end{rem}

\begin{thm}\label{thm:admissiblecuts}
Let $K$ be an algebraically closed field and let $\Gamma^{\times}=(\Gamma'_0, \Gamma_0^\times, \Gamma_1, \mf, \of)$ be a skew-Brauer graph with $\mf(v)=1$ for every vertex $v$ in $\Gamma_0$ and $B_{\Gamma^\times}=KQ_{\Gamma^{\times}}/I_{\Gamma^{\times}}$ be its associated skew-Brauer graph algebra. Then there exists (not necessarily unique) skew-gentle algebra $A=\SG$, such that $T(A^{\mathrm{sg}})\simeq B_{\Gamma^{\times}}$.
\end{thm}

\begin{proof}
Let $\Gamma^{\times}=(\Gamma'_0, \Gamma^{\times}_{0}, \Gamma_1, \mf, \of)$ be the skew-Brauer graph with $\mf(v)=1$ for every vertex $v$ in $\Gamma_0$ and $B_{\Gamma^\times}$ be a skew-Brauer algebra with $\mathcal{C}$ the set of sg-special cycles and $\Qsp$ the set of special vertices.

Consider $B_{\Gamma}$ be the auxiliary Brauer graph algebra, with $\Gamma=(\Gamma'_0\cup \Gamma^{\times}_{0}, \Gamma_1, \mf, \of)$.
Notice that $Q_\Gamma$ is also obtained from $Q_{\Gamma^{\times}}=(Q_\Gamma, \Qsp)$ by forgetting the extra information of $\Qsp$.
We claim that any admissible cut of $B_{\Gamma}$ induces a cut set of $B_{\Gamma^{\times}}$.

By \cite{Sch15}, there exists an admissible cut $\mathcal{D}$ on $B_{\Gamma}$ such that $B_{\Gamma}/\langle\mathcal{D}\rangle$ is gentle. Let $\mathcal{D}'$ be the set of arrows $$\mathcal{D}'=\{{}^{\varepsilon}\alpha^{\varepsilon'} \mid \alpha\in \mathcal{D} \quad\text{for all possible }\quad \varepsilon, \varepsilon'\in\{+,-, \emptyset \}\}$$ in $B_{\Gamma^{\times}}$.

Recall that by definition of $B_{\Gamma^{\times}}$, the set of sg-special cycles $\mathcal{C}$ are induced by the set of special cycles of $B_{\Gamma_A}$. To fix notation, for each special cycle $C_v$ of  $Q_\Gamma$, denote by  $\{C_{v, 1}, \dots C_{v, 2^r}\}$ the set of  $2^r$ sg-special cycles in $Q_{\Gamma^\times}^{\mathrm{sg}}$ (up to cyclic equivalence).

Suppose that there is a sg-special cycle $C_{v,l}= {}^{\varepsilon_1}\alpha_1^{\varepsilon_1'} \dots {}^{\varepsilon_m}\alpha_m^{\varepsilon_m'}$ that contains two different arrows ${}^{\varepsilon_i}\alpha_i^{\varepsilon_i'}$ and ${}^{\varepsilon_j}\alpha_j^{\varepsilon_j'}$ of $\mathcal{D}'$. Then by construction the cycle $C_v=\alpha_1\dots\alpha_m$ is special and $\alpha_i$ and $\alpha_j$ are arrows of $\mathcal{D}$. Since $\mathcal{D}$ is an admissible cut, then $\alpha=\alpha_i=\alpha_j$, and because $B_{\Gamma}/\langle \mathcal{D}\rangle$ is gentle $\alpha$ appears exactly once in $C_v$, this implies that ${}^{\varepsilon_i}\alpha_i^{\varepsilon_i'}={}^{\varepsilon_j}\alpha^{\varepsilon_j'}_i$.

To finish the proof, it is enough to observe that the algebra $B_{\Gamma^\times}/\langle \mathcal{D}'\rangle$ is the admissible version of $B_{\Gamma}/\langle \mathcal{D} \rangle$ with $\Qsp$ the set of special vertices of $B_{\Gamma^{\times}}$. \end{proof}  

We introduce the following definition motivated by the cut sets of a skew-Brauer algebra $B_{\Gamma^\times}$ used in Theorem~\ref{thm:admissiblecuts}, which are induced by cut sets of its auxiliary Brauer algebra $B_{\Gamma}$.

\begin{defn}\label{def:good cut}
Let $\Gamma^{\times}=(\Gamma'_0, \Gamma_0^\times, \Gamma_1, \mf, \of)$ be a skew-Brauer graph with $\mf(v)=1$ for every $v$ in $\Gamma_0$ and $B_{\Gamma^\times}=KQ_{\Gamma^{\times}}/I_{\Gamma^{\times}}$ be its associated skew-Brauer graph algebra. A cut set $\mathcal{D}$ is \emph{good} if it is induced by a cut set $\mathcal{D'}$ of its auxiliary Brauer graph algebra $B_{\Gamma}$, namely $$\mathcal{D}=\{{}^{\varepsilon}\alpha^{\varepsilon'}\mid \alpha\in \mathcal{D'} \quad\text{for all possible }\quad \varepsilon, \varepsilon'\in\{+,-, \emptyset \}\}$$
\end{defn}

We observe that any good cut set of a skew-Brauer graph algebra induces a skew-gentle algebra. We show that converse is not true in the following example.

\begin{exmp}\label{ex:admissible cut}
In this example, we show a cut set $\mathcal{D}$ of skew-Brauer algebra that is not good, but the quotient algebra induced by $\mathcal{D}$ is skew-gentle.

Let $B=kQ_{\Gamma^\times}/I_{\Gamma^\times}$ be a skew-Brauer algebra with quiver $Q_{\Gamma^{\times}}$ and skew-Brauer graph $\Gamma^{\times}$ depicted below 

\begin{figure}[h!]
   \adjustbox{scale=.9,center}
{\begin{tikzcd}
& & & & & & & & \\
Q_{\Gamma^{\times}}: &1^{\times} \arrow[r, shift left=.75ex, "\alpha"] &2 \arrow[r, shift left=.75ex, "\beta"] \arrow[l, shift left=.75ex, "\gamma"]&3^\times  \arrow[l, shift left=.75ex, "\delta"] && \Gamma^{\times}: \quad v_1^\times \arrow[r, dash] & v_2 \arrow[r, dash] & v_3 \arrow[r, dash] &v_4^\times. \\
& & & & & & & &
\end{tikzcd}
}
\end{figure}
To compute the cut set, see that the admissible presentation $B^{\mathrm{sg}}=KQ^{\mathrm{sg}}_{\Gamma^\times}/I^{\mathrm{sg}}_{\Gamma^{\times}}$ of the algebra $B$ is given by quiver $Q^{\mathrm{sg}}_{\Gamma^{\times}}$

\begin{figure}[h!]
   \adjustbox{scale=.9,center}
{\begin{tikzcd}
    1^+  \arrow[dr, shift left=.75ex, "{}^+\alpha"] & & 3^+ \arrow[dl, shift left=.75ex, "{}^+\delta"]\\
        &2 \arrow[ur, shift left=.75ex, "\beta^+"] \arrow[dr, shift left=.75ex, "\beta^-"] \arrow[ul, shift left=.75ex, "\gamma^+"] \arrow[dl, shift left=.75ex, "\gamma^-"]&\\
    1^- \arrow[ur, shift left=.75ex, "{}^-\alpha"] & & 3^- \arrow[ul, shift left=.75ex, "{}^-\delta"]    
\end{tikzcd}
}
\end{figure}
with sg-special cycles $({}^+\alpha)(\gamma^{+}), ({}^-\alpha)(\gamma^{-}), (\beta^+)({}^+\delta),  (\beta^-)({}^-\delta)$ and their permutations, and the ideal $I_{\Gamma^{\times}}^{\mathrm{sg}}$ is generated by all relations

\begin{enumerate}

\item $C-C'$, where $C$ and $C'$ are sg-special cycles that start at the same vertex, for example $(\gamma^{+})({}^+\alpha)- (\beta^-)({}^-\delta)$;
\item the paths of length 2 that start and end in two different vertices of the set $\{1^+, 1^-, 3^+, 3^-\}$, for example $({}^+\alpha)(\beta^-{})$;
\item the paths of length 3 of the form $C\eta$, where $C$ is a sg-special cycle, and $\eta$ is the first arrow of $C$, for example, $(\gamma^{+})({}^+\alpha)(\gamma^{+})$;

\item and finally, the paths of length 3 of the form $(\eta_1^{\varepsilon})({}^{\varepsilon}\eta_2)(\eta_1^{\varepsilon'})$ where $(\eta_1^{\varepsilon})({}^{\varepsilon}\eta_2)$ is a sg-special cycle that starts at 2,  $\varepsilon, \varepsilon'$ are either $+,-$ and $\varepsilon\neq \varepsilon'$, for example $(\gamma^{+})({}^+\alpha)(\gamma^{-})$.
\end{enumerate}

Consider $\mathcal{D}=\{{}^+\alpha, \gamma^-,{}^+\delta, \beta^-\}$ a cut set  of $B^{\mathrm{sg}}$. The quotient algebra of $B^{\mathrm{sg}}$ induced by $\mathcal{D}$ is the algebra $A=KQ/I$  where the quiver $Q$ is

\begin{figure}[h!]
   \adjustbox{scale=.9,center}
{\begin{tikzcd}
    1^+   & & 3^+ \\
        &2 \arrow[ur, "\beta^+"]  \arrow[ul, swap, "\gamma^+"] &\\
    1^- \arrow[ur, swap, "{}^-\alpha"] & & 3^- \arrow[ul,  "{}^-\delta"]    
\end{tikzcd}
}
\end{figure}

and $I$ is the ideal generated by the relations: 
\[({}^-\alpha)(\gamma^{+}), ({}^-\alpha)(\beta^{+}), ({}^- \delta)(\gamma^+), ({}^-\delta)(\beta^+).\]

Finally, $A$ is isomorphic to the skew-gentle algebra $C=(\widetilde{Q}, \widetilde{I}, \widetilde{\textrm{Sp}})$, where $\tilde{Q}$ is the quiver

\[\begin{tikzcd}
     &  &   \\
1 \arrow[r,swap,"\alpha"] \arrow[out=60,in=120,loop  , distance=1.6em,swap, "f_1"] & 2 \arrow[r,swap, "\beta"] & 3\arrow[out=60,in=120,loop  , distance=1.6em, swap, "f_2"]
\end{tikzcd}\]
$\tilde{I}=\langle\alpha\beta\rangle$ is the ideal and $\widetilde{\textrm{Sp}}=\{f_1, f_2\}$ is the set of special loops.
\end{exmp}

\section{Admissible cuts, the repetitive algebra and reflections}\label{sec:reflections}

In this section we consider the non-admissible presentation of a skew-gentle algebra, and we describe the quiver of its repetitive algebra. We also describe the positive and negative reflections on a skew-gentle algebra. Moreover, we analyse the  good cuts over a skew-gentle algebra and the relationship between these two notions.

The following definition of repetitive algebra for skew-gentle algebras, in their non-admissible presentation, is an extension of the one given for basic algebras by Hughes and Waschb{\"u}sch.

\begin{defn} 
Let $A=\SG$ be a skew-gentle algebra, $\Qsp$ be the set of special loops and $\mathcal{M}$ be the set of sp-maximal paths in $A$. The repetitive algebra of $A$ is $\widehat{A}=K\widehat{Q}/\widehat{I}$ where $\widehat{Q}$ and $\widehat{I}$ are defined as follows
\begin{itemize}
    \item The set of vertices of $\widehat{Q}$ is the set $\{x[n]: x\in Q_0 \text{ and } n\in \mathbb{Z}\}$.
    We say that $x[n]$ is a special vertex in $\widehat{Q}$ if an only if $x$ is a special vertex in $Q$ and $n\in \mathbb{Z}$.
    \item The set of arrows $\widehat{Q}_1$ is the set $\{\alpha[n]: x[n]\to y[n]: \alpha\in Q_1\text{ and } n\in\mathbb{Z}\}$.
    \item Each $sp$-maximal path $p: x\to y$ induces a set of \textit{connecting arrow} $\overline{p}[n]: y[n]\to x[n+1]$ for all $n\in \mathbb{Z}$.
\end{itemize}
To define the set of relations $\widehat{\rho}$ of $\widehat{Q}$, we need to introduce the following definition. If $p=p_qp_2$ is a sg-maximal path in $A$, then $p_2[n]\overline{p}[n]p_1[n+1]$ is a sg\emph{-full path} in $\widehat{Q}$.

The set of relations $\widehat{\rho}$ of $\widehat{Q}$ are given by the following list.
\begin{enumerate}
    \item Each monomial relation $p\in\rho$ induces a path $p[n]$ in $\widehat{\rho}$, for all $n\in\mathbb Z$.
    \item Each relation $f^2-f$ in $\rho$ induces a relation $f^2[n]-f[n]$  in $\widehat{\rho}$, for all $n\in\mathbb Z$.
    \item Any path $p$ in $\widehat{Q}$ which contains a connecting arrow is an element of $\widehat{\rho}$ if $p$ is nor a sub-path of a sg-full subpath.
\end{enumerate}
\end{defn}

\begin{lem}\label{quivers algebras repetitivas}
Let $A=\SG$ be a skew-gentle algebra and denote by  $A^{\mathrm{sg}}$  its admissible presentation. Let $A'$ be the auxiliary gentle algebra of $A$. Then, the quiver $Q_{\widehat{A}^{\mathrm{sg}}}$ of the repetitive algebra of $A^{\mathrm{sg}}$ coincides with the $\mathrm{sg}$-quiver $Q^{\mathrm{sg}}_{\widehat{A'}}$ of the repetitive algebra of $A'$, considering the set of special vertices in $Q_{\widehat{A'}}$ as the set $\{x[n]: x\in\Qsp \text{ and } n\in\mathbb{N}\}$.
\end{lem}

\begin{proof}
 It follows from Lemma \ref{lemma:quiver of trivial extension} and that the sg-maximal paths in $A$ are in correspondence with maximal paths in $A'$.
\end{proof}

\begin{exmp}\label{ex:repetitive algebra}
Let $A=\SG$ be the skew-gentle algebra with quiver $Q$
\[\begin{tikzcd}
1 \arrow[out=120,in=60,loop, looseness=4, "f_1"] \arrow[r, "\alpha"]& 2 \arrow[r,"\beta"]  & 3 \arrow[out=120,in=60,loop,looseness=4, "f_2"] \end{tikzcd}\]
ideal $I= \langle \alpha \beta \rangle$ and set of special vertices $\Qsp=\{1,3\}$.

To compute the repetitive algebra of $A$, we need to find the sg-maximal paths in $A$. In this case, there are exactly two sg-maximal path, namely $p_1=f_1\alpha$ and $p_2=\beta f_2$. Then by the definition, the quiver $Q_{\widehat{A}}$ of repetitive algebra of the algebra $A$ is the following.

{\tiny{\[\begin{tikzcd}
\dots\arrow[r, out=80, in=125, "\overline{p_1}"]  & 1[-1] \arrow[out=110,in=60,loop, looseness=4, "f_1"] \arrow[r, "\alpha"]& 2[-1] \arrow[r,"\beta"] \arrow[rr, out=60, in=125, "\overline{p_1}"]  & 3[-1] \arrow[out=110,in=60,loop,looseness=4, "f_2"] \arrow[rr, out=-60, in=-110, swap,"\overline{p_2}"] & 
1[0] \arrow[out=110,in=60,loop, looseness=4, "f_1"] \arrow[r, "\alpha"]& 2[0] \arrow[rr, out=60, in=125, "\overline{p_1}"] \arrow[r,"\beta"]  & 3[0] \arrow[rr, out=-60, in=-110, swap,"\overline{p_2}"]\arrow[out=110,in=60,loop,looseness=4, "f_2"]  &
1[1]  \arrow[out=110,in=60,loop, looseness=4, "f_1"] \arrow[r, "\alpha"]& 2[1] \arrow[r,"\beta"]  & 3[1]\arrow[out=110,in=65,loop,looseness=4, "f_2"] \arrow[r, out=-60, in=-110, swap,"\overline{p_2}", no head] & \dots 
\end{tikzcd}\]}}

While the quiver of the repetitive algebra of the admissible presentation $A^{\mathrm{sg}}$ of $A$
{\tiny
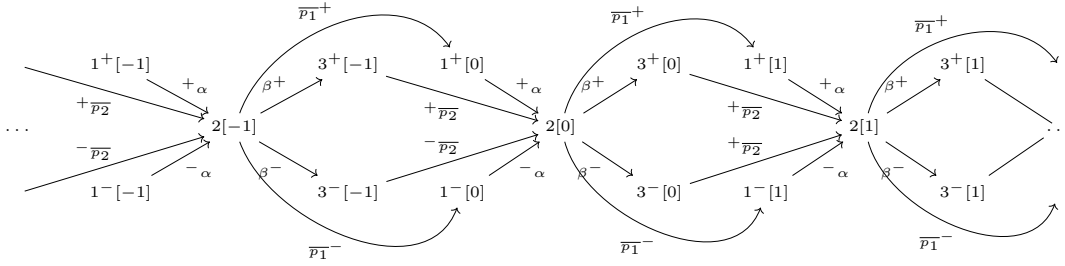
{
\[
  \begin{tikzcd}
  {} \arrow[rrd,swap, "{}^+\overline{p_2}"]&1^+[-1] \arrow[rd, "{}^+\alpha"] & & 3^+[-1]  \arrow[rrd, swap, "{}^+\overline{p_2}"] & 1^+[0] \arrow[rd, "{}^+\alpha"]
   & & 3^+[0] \arrow[rrd,swap, "{}^+\overline{p_2}"]  & 1^+[1] \arrow[rd, "{}^+\alpha"]
   & &  3^+[1] \arrow[rd, no head] &{} 
    \\
 \dots & &2[-1] \arrow[rrd, out=-70, in=-110, swap,"\overline{p_1}^-"]\arrow[rru, out=70, in=125, "\overline{p_1}^+"] \arrow[ru, "\beta^{+}"] \arrow[rd, swap,"\beta^{-}"] &&&2[0] \arrow[rrd, out=-75, in=-110, swap,"\overline{p_1}^-"] \arrow[rru, out=75, in=125, "\overline{p_1}^+"]  \arrow[ru, "\beta^{+}"] \arrow[rd, swap,"\beta^{-}"] &&&2[1] \arrow[rrd, out=-75, in=-110, swap,"\overline{p_1}^-"] \arrow[rru, out=75, in=125, "\overline{p_1}^+"]  \arrow[ru, "\beta^{+}"] \arrow[rd, swap,"\beta^{-}"] &&\dots\\
{} \arrow[rru, "{}^-\overline{p_2}"]&1^-[-1] \arrow[ru,swap, "{}^-\alpha"]  & & 3^-[-1] \arrow[rru,"{}^-\overline{p_2}"] & 1^-[0] \arrow[ru, swap, "{}^-\alpha"]
   &&  3^-[0] \arrow[rru,"{}^+\overline{p_2}"] & 1^-[1] \arrow[ru, swap, "{}^-\alpha"]
   && 3^-[1] \arrow[ru, no head] &{}
  \end{tikzcd}
   \]}}

\noindent One can check that the set of relations coincides with the relations induced by the sg-ideal of the repetitive algebra $\widehat{A}$.
\end{exmp}

Next result shows when reflections on a skew-gentle algebra return a skew-gentle algebra.

\begin{thm}\label{thm: refl-skew}
Let $A=\SG$ be a skew-gentle algebra and $x$ be a source in its auxiliary gentle algebra $A'$.
If $x$ is a special vertex (resp. non special vertex), then the negative reflection sequence $S^-_{x^+}S^-_{x^-}(A^{\mathrm{sg}})$ (resp. the negative reflection $S^-_{x}(A^{\mathrm{sg}})$) is skew-gentle.

A similar result follows when $x$ is a sink in $A'$ if we consider positive reflections.
\end{thm}

\begin{proof}
    Let $x$ be a source of $Q_0$. If $x\in \Qsp$ or $x\notin \Qsp$, in both cases it follows from Lemma \ref{quivers algebras repetitivas} that the sg-quiver $(\sigma_xQ_{A'})^{\mathrm{sg}}$ coincides with the quivers $\sigma_{ x^+}\sigma_{x^-}Q_{A^{\mathrm{sg}}}$ and $\sigma_{ x}Q_{A^{\mathrm{sg}}}$, respectively. 
    It is easy to see that the algebra $S^-_x(A')$ is a gentle algebra, where $x$ is a sink vertex for its ordinary quiver.  Therefore the sg-algebra $(S^-_x(A'))^{\mathrm{sg}}$  is skew-gentle, which is isomorphic to  $S_{x^+}S_{x^-}(A)$ and  $S_{x}(A)$, respectively.
\end{proof}

The following example shows that certain sequences of reflections over a skew-gentle algebra, not defined as in the previous theorem, can also result on a skew-gentle algebra.

\begin{exmp}
Let $A=\SG$ be the skew-gentle algebra from Example~\ref{ex:repetitive algebra}.

By the Theorem \ref{thm: refl-skew}, the algebras $S_{1^{+}}^-S_{1^-}^-(A^{\mathrm{sg}})$ and $S_{3^{+}}^+S_{3^-}^+(A^{\mathrm{sg}})$ are skew-gentle. We note that the algebra $S_{3^-}^+S_{1^-}^-(A^{\mathrm{sg}})$ is also skew-gentle, but  it is not obtained by a sequence of reflections as described in Theorem \ref{thm: refl-skew}. 
\end{exmp}

To show the link between reflections and admissible cuts, we need to use the following result, which is part of a work in progress \cite{ABFGMRT24}. For the benefit of the reader, we include the proof of this result.

\begin{lem}[\cite{ABFGMRT24}]\label{lem:cuts}
    Let $A=KQ/I$ be a gentle algebra,  $T(A)$ its trivial extension, and $x$ a source vertex of $A$. Then, there exists an admissible cut $\mathcal{D}$ of $B$ such that:
    \begin{enumerate}
        \item $S^-_x(A)$ is isomorphic to the quotient $T(A)/\langle \mathcal D \rangle$.
        \item the admissible cut $\mathcal{D}$ contains the arrows in $Q$ that start at $x$; in other words,  $\mathcal{D}\cap Q_1=\{\alpha\in Q_1\mid s(\alpha)=x\}$.
    \end{enumerate}
\end{lem}

\begin{proof}
Let $x$ be a source vertex in $A$. We denote by $B=S^-_x(A)$ the negative reflection of $A$.
By \cite{HW83}, we have that the trivial extension algebra $T(A)$ of $A$ coincides with the trivial extension algebra $T(B)$ of $B$, and moreover, by \cite{FSTTV22}, there exists an admissible cut $\mathcal{D}$ of $T(A)$ such that $B=T(A)/<\mathcal{D}>$.

To prove the second statement, let $\gamma: y\to z$ be an element of $\mathcal{D}\cap (Q)_1$. Suppose  $s(\gamma)\neq x$. By construction, $(\gamma,0):(y,0)\rightarrow (z,0)$ is an arrow of $\sigma^-_x(Q_{A})=Q_B$, contradicting the fact that $\gamma$ is in $\mathcal{D}$. Therefore $s(\gamma)=x$.

Conversely, if $\gamma$ is an arrow of $Q_1$ such that $s(\gamma)=x$, since the vertex $x$ is a sink of $Q$ we have that $\gamma\notin (Q_{B})_1$ and consequently $\gamma\in \mathcal{C}$.
\end{proof}

\begin{thm}\label{thm:cuts and reflections}
   Let $A=\SG$ be a skew-gentle algebra and let $x$ be a source vertex in the auxiliary gentle algebra $A'$ of $A$. If $x$ is a special vertex (resp. $x$ is not a special vertex), then, $S^{-}_{x^+}S^{-}_{x^-}(A^{\mathrm{sg}})=T(A^{\mathrm{sg}})/\langle\mathcal{D}\rangle$ (resp. $S^{-}_{x}(A^{\mathrm{sg}})=T(A^{\mathrm{sg}})/\langle\mathcal{D}\rangle$) where $\mathcal{D}$ is the good cut of $T(A^{\mathrm{sg}})$ induced by the cut set $\mathcal{D}'$ such that $$\mathcal{D}'\cap (Q_{A'})_1=\{\alpha \in (Q_{A'})_1\mid s(\alpha)=x\}.$$
\end{thm}
\begin{proof}
Let $x$ be a source vertex in $A'$ and suppose that $x$ is a special vertex of $A$. We denote by $B=S^-_{x^+}S^-_{x^-}(A^{\mathrm{sg}})$ the negative reflection sequence of $A^{\mathrm{sg}}$. By Theorem \ref{thm: refl-skew} we know that $B$ is a skew-gentle algebra and $B'=S^{-1}_x(A')$ is the auxiliary gentle algebra of $B$.

By Lemma~\ref{lem:cuts}, there exists an admissible cut $\mathcal{D}$ of $A'$ such that $\mathcal{D}'\cap (Q_{A'})_1=\{\alpha \in (Q_{A'})_1\mid s(\alpha)=x\}$, which induces a good cut set $\mathcal D$ in $T(A^{\mathrm{sg}})$. The result follows from Theorem~\ref{thm: refl-skew}.\end{proof}
   
\section{Skew-Brauer graph algebras of finite representation type}\label{sec:finrep}

In this section, we show that the class of skew-Brauer graph algebras of finite representation type are characterized by the structure of the skew-Brauer graph as a graph and the number of special vertices on it. {\it In the remainder of this section, we consider non-trivial skew-Brauer graphs $\Gamma^{\times}$, that is, with at least one special vertex.}

\begin{defn}
A skew-Brauer graph $\Gamma^\times=( \Gamma'_0, \Gamma_0^{\times}, \Gamma_1,  \mf, \of)$, with $\Gamma_0^{\times}\neq \emptyset$, is a \emph{skew-Brauer tree} if the underlying graph $(\Gamma'_{0}\cup\Gamma_0^\times, \Gamma_1)$ is a tree, $\mf(v)=1$ for all $v \in \Gamma'_0 \cup \Gamma_0^{\times}$, there is exactly one special vertex in $\Gamma^\times$, and $\Gamma$ has at least two edges. We say that $B_{\Gamma^{\times}}$ is a {\it skew-Brauer tree algebra} if $\Gamma^{\times}$ is a skew-Brauer tree.
\end{defn} 

We exclude from the previous definition the skew-Brauer graph consisting of exactly one edge and one special vertex, since in this case, $B_{\Gamma^{\times}}$ is isomorphic the direct sum two copies of $\mathbb{A}_1$.

Recall that a \emph{Brauer tree algebra} is a Brauer graph algebra defined by a graph $\Gamma$ that is a tree, with $\mf(v) = 1$ for all but at most one $v \in \Gamma_0$. Since skew-Brauer graph algebras are symmetric, we also recall the classification of symmetric finite representation type algebras.

\begin{lem}(\cite[Theorem 2.11]{aihara})\label{Thm: classification}
    A non-simple basic symmetric algebra is of finite representation type if and only if it is isomorphic to an algebra in the following mutually exclusive classes.

    \begin{enumerate}
        \item Trivial extensions of iterated tilted algebras of Dynkin type.
        \item Brauer tree algebras with exactly one vertex $v$ with multiplicity $\mf(v)>1$.
        \item Modified Brauer tree algebras with exactly one vertex $v$ with multiplicity $\mf(v)>1$.
\end{enumerate}
\end{lem}

In the previous lemma, the family of Brauer tree algebras with $\mf(v)=1$ for all $v$ coincides with the trivial extensions of iterated tilted algebras of Dynkin type $\mathbb{A}$, see \cite[Lemma 3.1]{WZ22}. The definition of Brauer graph algebras, Brauer tree algebras and modified Brauer tree algebras can be found in 
\cite[Section 2.3]{Sch18} and \cite[Section 2.8 and Section 3.6]{Sko2006}. We notice that modified Brauer tree algebras are not necessarily biserial, see for example the projective-injective $P_1$ in \cite[Example 3.7]{Sko2006}.

\begin{lem}\label{lema: distiguished vertices}
    Let $B_{\Gamma^{\times}}= KQ^{\mathrm{sg}}/I^{\mathrm{sg}}$ be a skew-Brauer graph algebra of finite representation type, where $\Gamma^{\times}$ is not a skew-Brauer graph with two edges. Then $\Gamma^{\times}$ has at most one special vertex. 
\end{lem}
\begin{proof}
    Let $B=B_{\Gamma^{\times}}$ be a skew-Brauer graph algebra of finite representation type. Suppose that $\Gamma^{\times}$ has at least two special vertices $v^{\times}$ and $w^{\times}$. By definition $\Gamma^{\times}$ is connected, then there exists a minimal path $\rho$ of edges in $\Gamma^{\times}$ that join $v^{\times}$ and $w^{\times}$ as depicted in the following figure.

    \begin{center}
        \begin{tikzcd}
            v^{\times} \arrow[r, no head,"a_1"] & x_1 \arrow[r, no head, "a_2"]& x_2 \arrow[r, no head, "a_3"]& \cdots \arrow[r, no head, "a_{k-1}"] & x_k \arrow[r, no head, "a_k"]& w^{\times}
        \end{tikzcd}
    \end{center}

We denote by $Q_{\Gamma}$ be the quiver of the auxiliary Brauer graph algebra. Recall that there is a unique special cycle $C_{x_i}$ in $Q_{\Gamma}$, up to cyclic equivalence, induced by the edges attached to the vertex $x_i\in \Gamma$, for $i=\{1, \dots, k\}$.

For each $i\in \{1, \dots, k\}$, we define a sub-path $\rho_i$ of the special cycle $C_{x_i}$ from the vertex  $a_{i+1}$ to $a_{i}$ if $i$ is odd or the sub-path from $a_{i}$ to $a_{i+1}$ if $i$ is even. Then the walk $\rho_1^{-1}\rho_2\dots \rho_k^{{(-1)}^{k}}$ defines a non-linearly oriented sub-quiver $Q'$ of type $A_n$ without relations in $Q_{\Gamma}$ such that the first and the last vertex are special. As a consequence $B$ contains an algebra of Dynkin type $\tilde{D}_n$ as a sub-algebra, which is a contradiction because $B$ is an algebra of finite representation type.\end{proof}

\begin{lem}\label{lema:quivers iguales}(\cite[II.3 Lemma 3.6]{assem2006}, \cite[Section 2.3]{Sch18} and \cite[Section 3.6]{Sko2006})  
\begin{enumerate}
    \item If two basic connected isomorphic algebras are defined by a bound quiver with an admissible ideal, then their quivers coincide.
   \item Let $KQ/I$ be a Brauer graph algebra defined by an admissible ideal, then for all $x\in Q_0$ we have that $x$ is the start (end) of exactly two arrows or $x$ is the start (end) of exactly one arrow.
   \item Let $KQ/I$ be a modified Brauer tree algebra defined by an admissible ideal, then for all $x\in Q_0$ we have that $x$ is the start (end) of exactly two arrows or $x$ the start (end) of exactly one arrow and $Q$ must contain exactly one loop.
\end{enumerate}
\end{lem}

It follows from the previous lemma that any bound quiver algebra $A$ cannot be isomorphic to a Brauer graph algebra or a modified Brauer graph algebra if its quiver $Q_A$ has three or more arrows ending at (or starting from) any of its vertices.

In the following lemma, we describe some properties of a subclass of skew-Brauer graph algebras that are relevant to this section.

\begin{lem}\label{lemma: excepcional cases}
Let $B=B_{\Gamma_i^{\times}}$ be a skew-Brauer graph algebra associated with one of the following skew-Brauer graphs.
    
\begin{center}
\begin{tikzcd}
\Gamma_1^\times: &   \times \arrow[r, no head]  & v \arrow[r, no head]  &    w & {} & {} & &\Gamma_2^{\times}:& \times \arrow[r, no head]  & v \arrow[r, no head]  &    \times
\end{tikzcd}
\end{center}

 \begin{enumerate}
    \item If $B$ is either  $B_{\Gamma_1^\times}$ where $\mf(w)=1$ or $B_{\Gamma_2^\times}$, then $B$ is isomorphic to a Brauer graph algebra (for one of the graphs $\Delta_{1,2}$ in Figure~\ref{fig:excepcional cases}).
    
    \item If $B=B_{\Gamma_1^\times}$ with $\mf(w)\neq 1$, then $B$ is not of finite representation type.
\end{enumerate}
\end{lem}

\begin{proof}
    By the definition of skew-Brauer graph algebras, the quiver $Q_B$ of $B$ is one of the quivers depicted in Figure \ref{fig:quivers}: the first two quivers correspond to $\Gamma_1^\times$ with $\mf(w)=1$ and $\mf(w)\neq 1$, respectively, while  the last one corresponds to $\Gamma_2^\times$ with arbitrary $\mf(v)$.
    \begin{figure}[ht!]
    \centering
\begin{tikzcd}
& & & & & & & & 1^+ \arrow[dd, shift right, "^+\beta^+", swap] \arrow[rr, shift right, "^+\beta^-", swap] && 2^- \arrow[dd, shift left, "^-\alpha^-"] \arrow[ll, shift right, "^-\alpha^+", swap]\\
 1^+ \arrow[r, shift right, "{}^+\beta", swap]  &  2\arrow[r, shift left, "\alpha^-"] \arrow[l, shift right, "\alpha^+", swap]  &   1^- \arrow[l, shift left, "{}^-\beta"]  & &1^+ \arrow[r, shift right, "{}^+\beta", swap]  &  2\arrow[r, shift left, "\alpha^-"] \arrow[l, shift right, "\alpha^+", swap] \arrow[loop above, "\gamma"] &   1^- \arrow[l, shift left, "{}^-\beta"]  && & &\\
   & & & & & & & & 2^+\arrow[rr, shift left, "^+\alpha^-"] \arrow[uu, shift right, "^+\alpha^+", swap] & &1^- \arrow[uu, shift left, "^-\beta^-"] \arrow[ll, shift left, "^-\beta^+"]
\end{tikzcd}
    \caption{Quivers of skew-Brauer graphs with underlying tree graph and 2 edges}
    \label{fig:quivers}
\end{figure}

    To prove the first statement, suppose that $B$ is either $B_{\Gamma_1^\times}$ with $\mf(w)=1$ or $B_{\Gamma_2^\times}$. Let $C$ be any cyclic permutation of the elementary cycle $({}^+\beta^{\epsilon})({}^{\epsilon}\alpha^+)$ or $({}^-\beta^{\epsilon})({}^{\epsilon}\alpha^-)$, for any possible $\epsilon\in\{+,-, \emptyset\}$, and let $m$ be a natural number greater than 1. It follows from relations of type 0 and relations of type III, that any $m$-power $C^m$ of the cycle $C$ is zero in $B$, and as a consequence one can assume that $\mf(v)=1$. Then we have that $B$ is isomorphic to a Brauer graph algebra $B_\Gamma$ with $\Gamma$ one of the Brauer graphs $\Delta_1$ or $\Delta_2$, respectively.

\begin{figure}[h!]
   \adjustbox{scale=.9,center}
{
\begin{tikzcd}
{}&{}&{}&{}&{}&v_1 \arrow[r, no head] \arrow[dd, no head] &v_2&\\ [-10pt] 
    v_1 \arrow[r, no head]  & v_2 \arrow[r, no head]  &    v_3 \arrow[r, no head] & v_4 & {} & {} & &    \\ [-15pt]
    {}&{}&{}&{}&{}&v_4 &v_3 \arrow[l, no head] \arrow[uu, no head]&
\end{tikzcd}
}
    \caption{Brauer graphs $\Delta_1$ (left) and $\Delta_2$ (right)}
    \label{fig:excepcional cases}
\end{figure}

with $\mf(v_i)=1$ for $i=1,2,3,4$.

Now, suppose that $B=B_{\Gamma_1^{\times}}$ where $\mf(w)\neq 1$. Then the algebra is given by the second quiver in Figure~\ref{fig:quivers} and the ideal 

\begin{equation*}
    \begin{split}
I_{B_{\Gamma_1^{\times}}}= &\langle (\alpha^{+})({}^+\beta)-(\alpha^{-})({}^-\beta), \gamma^{\mf(w)+1}, \gamma^{\mf(w)}-(\alpha^{-})({}^-\beta),  \gamma^{\mf(w)}-(\alpha^{+})({}^+\beta), \\ &({}^+\beta)\gamma, ({}^+\beta)(\alpha^{-}), ({}^-\beta)(\alpha^{+}),({}^-\beta)\gamma, \gamma(\alpha^{+}), \gamma(\alpha^{-})  \rangle.
\end{split}
\end{equation*}
We observe that $B$ is an algebra of infinite representation type, since there is band module associated to the word $\gamma^{-1}(\alpha^{+})({}^+\beta)$.
\end{proof}

\begin{lem}\label{lem:not brauer graph}
Let $B=B_{\Gamma^{\times}}$ be a skew-Brauer graph algebra with a non empty set $\Gamma_0^{\times}$ of special vertices
and suppose that $B$ is neither $B_{\Gamma_1^\times}$ where $\mf(w)=1$ nor $B_{\Gamma_2^\times}$, where $\Gamma_i^{\times}$ is one of the following two skew-Brauer graphs.

\begin{center}
\begin{tikzcd}
   \Gamma_1^{\times} \colon \times \arrow[r, no head]  & \bullet \arrow[r, no head]  &    w & {} & {} & \Gamma_2^{\times} \colon  \times \arrow[r, no head]  & \bullet \arrow[r, no head]  &    \times
\end{tikzcd}
\end{center}
Then $B$ is not a Brauer tree algebra.
\end{lem}

\begin{proof}
Let $\Gamma^\times$ be a skew-Brauer graph with a special vertex $v^\times$. By definition of skew-Brauer graph, there exists a non-special vertex $v\in\Gamma^{\times}$ such that $v$ and $v^\times$ have exactly one edge $x$ in common. We claim that if $v$ has valency at least three, $B_{\Gamma^{\times}}$ is not a Brauer graph algebra.

Since $v$ has at least valency three, consider $x$ and $y=\of^{-1}(x)$ two different edges attached to $v$, where $y$ might be a loop. Then by Remark~\ref{rem: projectives},  there are at most two projectives in $B_{\Gamma^\times}$ associated to $y$, namely $P_y$ or $P_{y^+}$ and $P_{y^-}$, in both cases the projectives are not biserial, see Example \ref{ex:projectives}. As a consequence $B_{\Gamma^\times}$ is not a Brauer graph algebra, and in particular is not a Brauer tree algebra.

Now, we suppose that $v$ has valency two. In this case, $x$ and  $y=\of^{-1}(x)$ are two different edges attached to $v$ and $y$ is not a loop, otherwise, the valency of $v$ would be grater than 2. We denote by $w$ the vertex in $\Gamma^{\times}$ such that $y$ is the edge connecting $v$ and $w$.

By hypothesis, the valency of $w$ is at least two, or if the valency of $w$ is one, then $\mf(w)\neq 1$. In the second case, $B$ is the skew-Brauer graph $B_{\Gamma_1^{\times}}$ with $\mf(w)\neq 1$, and by Lemma \ref{lemma: excepcional cases} is not of finite representation type, therefore can not be isomorphic to a Brauer tree algebra.

Now we suppose that $w$ has valency at least two. By the definition of the quiver $Q=Q_{\Gamma^\times}$,
the local configuration for the vertex $y \in Q$ associated to the edge $y$ in $\Gamma^{\times}$ (here $y$ receives the same name by abuse of notation) is as follows.

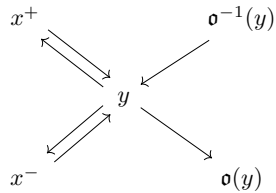
\begin{figure}[h!]
   \adjustbox{scale=.9,center}
{
\begin{tikzcd}
    x^+ \arrow[dr, shift left]  &  {}  &   \of^{-1}(y) \arrow[dl]\\
    & y \arrow[lu, shift left] \arrow[ld, shift right] \arrow[dr]& \\
   x^- \arrow[ur, shift right] & {} & \of(y) 
\end{tikzcd}
}
\end{figure}

where $\of(y)$ and $\of^{-}(y)$ might be equal. Then there are tree arrows ending at $y$ and, by Lemma \ref{lema:quivers iguales}, 
$B$ is not a Brauer graph algebra as we claim.\end{proof}

\begin{lem}\label{lem:not modified brauer graph}
Let $\Gamma^\times$ be a skew-Brauer graph with at least two edges and with a non-empty set $\Gamma_0^{\times}$ of special vertices. If $B_{\Gamma^{\times}}$ is of finite representation type, then it is not a modified Brauer tree algebra.
\end{lem}

\begin{proof}
Since $B_{\Gamma^{\times}}$ is of finite representation type, by Lemma \ref{lema: distiguished vertices}, $\Gamma^{\times}$ has exactly one special vertex $v^{\times}$. By definition of skew-Brauer graph, there is exactly one edge $x$ incident to  $v^\times$. We denote by $v$ the other vertex of $x$.

Consider the edge $y = \of(x)$ having $v$ and $u$ as vertices. This is the minimal configuration for a skew-Brauer graph with a special vertex. We analyse cases for the rest of the graph. 
\begin{center}
    \begin{tikzcd} 
v^\times \arrow[r, "x", no head] & v \arrow[r, "y", no head] & u
\end{tikzcd}
\end{center}

Case 1: suppose $\mf(u) > 1$ or $\mathrm{val}(u)>1$. Then the vertex $y \in Q_{\Gamma^{\times}}$ is the end of three arrows. 

Case 2: the following three $\mf(u) = 1$, $\mathrm{val}(u)=1$ and $\mathrm{val}(v)>1$ hold. Then there are exactly two arrows ending at $y \in Q_{\Gamma^{\times}}$ and one arrow starting at $y$.

Case 3: the graph $\Gamma^{\times}$ is the one depicted above. In that case either there is no loop in $Q_{\Gamma^{\times}}$ or the vertex $y$ is the start and end of three arrows, one of them being a loop.

Given the three cases, by Lemma \ref{lema:quivers iguales} (3) the algebra $B_{\Gamma^{\times}}$ is not a modified Brauer tree algebra.\end{proof}

So far we have established that, except for small pathological cases (see Lemma \ref{lemma: excepcional cases}), skew-Brauer graph algebras of finite representation type are not Brauer graph algebras nor modified Brauer tree algebras. So we should explore the existence of representation finite skew-Brauer graph algebras in the family of trivial extensions of iterated tilted algebras of type $\mathbb{D}$ or $\mathbb{E}$. The following result shows that any skew-Brauer graph algebra of finite representation type is not the trivial extension of a iterated tilted algebra of Dynkin type $\mathbb{E}$.

\begin{lem}\label{lem:not E}
Let $\Gamma^\times$ be a skew-Brauer graph with at least two edges and with a non-empty set $\Gamma_0^{\times}$ of special vertices. If $B_{\Gamma^{\times}}$ is of finite representation type, then $B_{\Gamma^{\times}}$ is not isomorphic to a trivial extension of an iterated tilted algebra of Dynkin type $\mathbb{E}$.
\end{lem}

\begin{proof}
    By Lemma \ref{lema: distiguished vertices}, the skew-Brauer graph $\Gamma^{\times}$ has exactly one special vertex, let $v^{\times}$ be such vertex and let $u$ be the non-special vertex of $\Gamma^{\times}$ that shares an edge with $v^{\times}$. Since $\Gamma^{\times}$ has exactly one special vertex, there are exactly two special cycles in $Q_{\Gamma^{\times}}$ induced by the edges attached to $u$. These cycles share all arrows except two, having the following diagram.

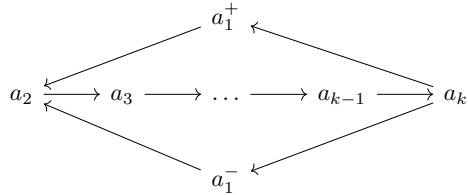
\begin{figure}[h!]
   \adjustbox{scale=.9,center}
{
    \begin{tikzcd}
        &&a_1^+\arrow{dll}&&\\
        a_2\arrow[r]&a_3\arrow[r]&\dots\arrow[r] &a_{k-1}\arrow[r]&a_k\arrow[llu] \arrow[lld]\\
        &&a_1^-\arrow[llu]&&
    \end{tikzcd}
}
\end{figure}
    
We notice that vertices $a_1^+$ and $a_1^-$  have exactly one arrow ending and one arrow starting at them, a property that is not allowed in the quivers of trivial extension of iterated tilted algebras of Dynkin type $\mathbb{E}$, see \cite[Section 3]{F99} for a complete list of quivers of trivial extension of iterate tilted algebras of Dynkin type.
\end{proof}

\begin{rem}\label{rem:cut set}
   In general, computing cut sets (or admissible cuts) is not an easy task, as the computations depend on the set of elementary cycles of the trivial extension of an algebra. However, in the context of trivial extension of gentle algebras, if we consider its Brauer graph, any set of arrows $\mathcal{D}$ in $Q_{\Gamma}$ containing exactly one arrow of each special cycle is a cut set, and furthermore $T(B/ \langle D \rangle)$ is isomorphic to $B$, see \cite[Theorem 1.3]{Sch15}.
\end{rem}

\begin{thm}\label{thm:finite representation type}
Let $\Gamma^{\times}$ be a skew-Brauer graph with at least two edges and with $\Gamma_0^{\times}\neq \emptyset$. If the  skew-Brauer graph algebra $B_{\Gamma^{\times}}$ is of finite representation, then:

\begin{enumerate}
    \item[$(1)$] $B_{\Gamma^{\times}}$ is isomorphic to a trivial extension of iterated tilted algebra of Dynkin type $\mathbb{A}_3$ or $\mathbb{D}_n$ for some natural number $n\geq4$, and 
    \item[$(2)$] $\Gamma^{\times}$ is a skew-Brauer tree graph.
\end{enumerate}
\end{thm}

\begin{proof}
$(1)$ Let $\Gamma^\times=(\Gamma'_0, \Gamma_0^{\times}, \Gamma_1,  \mf, \of)$ be a skew-Brauer graph, let  $\Gamma=(\Gamma'_{0}\cup\Gamma_0^\times, \Gamma_1,  \mf, \of)$ be its auxiliary Brauer graph, and $\widehat{\Gamma}=(\Gamma'_{0}\cup\Gamma_0^\times, \Gamma_1)$ be its underlying graph. We now analyze $B=B_{\Gamma^{\times}}$ based on
the number of edges in $\Gamma$.

Suppose that $\Gamma^{\times}$ has exactly two edges, then by Lemma~\ref{lemma: excepcional cases}, $\Gamma^{\times}=\Gamma_1^{\times}$ with multiplicity function identically one. Even more, $B_{\Gamma^\times}$ is isomorphic to the Brauer tree algebra associated with the Brauer tree shown in Figure~\ref{fig:excepcional cases}, and in this case, $B_\Gamma$ is isomorphic to the trivial extension of the iterated tilted algebra of Dynkin type $\mathbb{A}_3$, see \cite[Theorem 1.3]{Sch15} and \cite[Lemma 3.1 and Theorem 3.2]{WZ22}.

Now, we assume $\Gamma^{\times}$ has more than 2 edges. In particular, this assumption implies that $\Gamma^{\times}$ is not $\Gamma_1$ nor $\Gamma_2$ described in Lemma~\ref{lemma: excepcional cases}. 
Since $B_{\Gamma^\times}$ is of finite representation type, Lemma~\ref{lema: distiguished vertices} implies that the skew-Brauer graph $\Gamma^{\times}$ has exactly one special edge.

By Lemma~\ref{lem:not brauer graph} we have that $B_{\Gamma^\times}$ is not a Brauer tree algebra (hence it is not isomorphic to the trivial extension of an iterated tilted algebra $A$ of Dynkin type $\mathbb{A}$.). By Lemma \ref{lem:not modified brauer graph} we know that $B_{\Gamma^\times}$ is not a modified Brauer tree algebra. And by Lemma \ref{lem:not E}, $B_{\Gamma^\times}$ is not a trivial extension of an iterated tilted algebra of type $\mathbb{E}_{6,7,8}$. Thus, the classification theorem in Lemma~\ref{thm: skew brauer simetricas} implies that $B_{\Gamma^\times}$ is isomorphic to the trivial extension of an iterated tilted algebra $A$ of Dynkin type $\mathbb{D}$.

$(2)$ To prove that $\Gamma^\times$ is a skew-Brauer tree, we first observe that the multiplicity function $\mf$ of $\Gamma^{\times}$ is identically equal to one. This follows because $B = B_{\Gamma^\times}$ is isomorphic to the trivial extension of an iterated tilted algebra $A$ of finite representation type, and by \cite[Theorem 2]{Yam80}, the quiver $Q_A$ does not have an oriented cycle. 

We now prove that the underlying graph $\widehat{\Gamma} = (\Gamma'_{0}\cup\Gamma_0^\times, \Gamma_1)$ is in fact tree. Suppose that $\widehat{\Gamma}$ is not a tree, then there exists a cycle $\mathcal{C}$ in $\widehat{\Gamma}$ 
 \begin{center}
        \begin{tikzcd}
            x_1\arrow[r, no head,"a_1"] & x_2 \arrow[r, no head, "a_2"]& x_3 \arrow[r, no head, "a_3"]& \cdots \arrow[r, no head, "a_{k-1}"] & x_k \arrow[r, no head, "a_k"] &x_{k+1}=x_1
        \end{tikzcd}
    \end{center}
such that $x_i\neq x_j$ if $i\ne j$ and $a_{k}\neq a_1$. We denote by $V=\{x_1, \dots, x_k\}$ the set of vertices of $\mathcal{C}$.
    
We notice that for each vertex $x_i$ there is exactly one special cycle $C_{x_i}$, up to cycle equivalence, in $Q_{\Gamma}$, given by the edges attached to each $x_i$,  for $i=1\dots, k$. Since there are no repetition of vertices in $\mathcal{C}$ and each arrow of $Q_{\Gamma}$ belongs to exactly one especial cycle, then by Remark~\ref{rem:cut set} the set $$\mathcal{D}=\bigcup\limits_{i=1}^{k}\{\alpha_{x_i}\in Q_{\Gamma}\mid \alpha_{x_i}: a_i \to b_i \}\bigcup \mathcal{D}_2$$ is a cut set of $Q_{\Gamma}$, where $\mathcal{D}_2$ is any set of arrows containing exactly one arrow $\alpha_x$ that is not a loop of each special cycle $C_x$ with $x\notin V$, and $\of(x_{i})=(..., a_{i}, b_i,\dots )$ up to rotation.
By Theorem~\ref{thm:admissiblecuts} and Theorem~\cite[Theorem 3.6]{FSTTV22} there exists a good cut $\mathcal{D}'$ of $Q_B$ induced by $\mathcal{D}$ such that $T(B/\langle\mathcal{D}'\rangle) \simeq B$. 

Finally, to prove the claim, observe that the quiver $Q_{\Gamma}$ has an oriented cycle induced by the cycle $\mathcal{C}$ in $\Gamma$ given by the sub-paths of $C_{x_i}$ from the vertex $a_{i+1}$ to the vertex $a_i$ in $Q_A$ for $i=2, \dots, k$ and the sub-path of $C_{x_1}$ from the vertex $a_1$ to $a_k$, and therefore $B/\langle \mathcal{D}'\rangle$  also has an oriented cycle.
This contradicts \cite{Yam80} as $B$ is a trivial extension of an iterated tilted algebra of Dynkin type $\mathbb{D}_n$.\end{proof}

\begin{thm}\label{thm:skew-brauer tree algebras}
Any skew-Brauer tree algebra is of finite representation type.
\end{thm}

\begin{proof}
Let $\Gamma^{\times}$ be a skew-Brauer tree. If $\Gamma^{\times}$ has exactly two edges, using similar arguments as in the proof of Theorem~\ref{thm:finite representation type} , $B_{\Gamma^{\times}}$ is isomorphic to the trivial extension of type $\mathbb{A}_3$.

Now suppose that $\Gamma^{\times}$ has more than two edges. By Theorem~\ref{thm: equiv1}, $B=B_{\Gamma^{\times}}$ is a trivial extension of a skew-gentle algebra. We will show that $B$ is actually a trivial extension of a iterated tilted algebra of Dynkin type $\mathbb{D}_n$.

Let $v^{\times}$ be the unique special vertex in $\Gamma^{\times}$, $u$ be the unique non-special vertex in $\Gamma^{\times}$ that shares an edge with $v^{\times}$, and $u_2, \dots, u_k$ be the non-special vertices in $\Gamma^{\times}$ that share an edge with $u$. Denote these last edges by $a_2, \dots, a_k$ respectively, see the picture below.

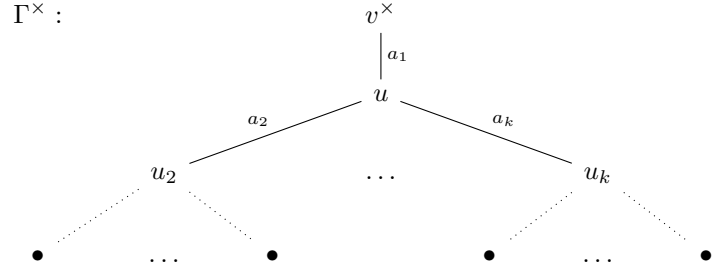
\begin{center}
    \begin{tikzcd}
  \Gamma^{\times}:&&&v^{\times}\arrow[d, no head, "a_1"]&&\\
    &&&u\arrow[dll, no head,swap, "a_2"] \arrow[drr, no head, "a_k"]&&\\
   &u_2\arrow[dl, no head, dotted] \arrow[dr, no head, dotted]&&\dots&&u_k \arrow[dr, no head, dotted] \arrow[dl, no head, dotted]\\
   \bullet&\dots& \bullet&&\bullet&\dots&\bullet
    \end{tikzcd}
\end{center}

Since the underlying graph of $\Gamma^{\times}$ is a tree, the subgraph obtained by deleting the vertex $u$ is disconnected, consisting of connected components $\Gamma_1, \Gamma_2, \dots, \Gamma_k$. These components are in bijection with the vertices $v^{\times}, u_2, \dots, u_k$, where each $u_i$ is a vertex of the connected component $\Gamma_i$ and $\Gamma_i$ is a tree for $i = 2, \dots, k$. Furthermore, $\Gamma_1$ is associated with the special vertex $v^{\times}$.

By \cite[Lemma 3.1]{WZ22}, each connected component $\Gamma_i$ for $i=2, \dots, k$ induces a trivial extension of an iterated tilted algebra of Dynkin type $\mathbb{A}_n$. Then any cut set $\mathcal{C}_i$ of $\Gamma_i$ defines an iterated tilted algebra of Dynkin type $\mathbb{A}_n$

Even more, by construction of $Q_{\Gamma^{\times}}$,  there are exactly two special cycles in $Q_{\Gamma^{\times}}$ induced by the edges attached to $u$ that share all arrows except two, having the following diagram.

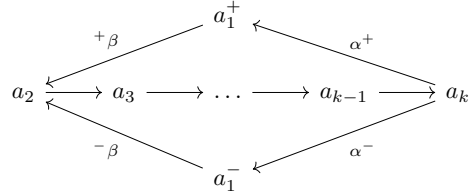
\begin{figure}[h!]
   \adjustbox{scale=.9,center}
{
\begin{tikzcd}
        &&a_1^+\arrow[dll, swap,"{}^+\beta"]&&\\
        a_2\arrow[r]&a_3\arrow[r]&\dots\arrow[r] &a_{k-1}\arrow[r]&a_k\arrow[llu, swap, "\alpha^+"] \arrow[lld, "\alpha^-"]\\
        &&a_1^-\arrow[llu, "{}^-\beta"]&&
\end{tikzcd}
}
\end{figure}

It is clear that the set  $\mathcal{C}=\bigcup\limits_{i=2}^k \mathcal{C}_i \cup \{\alpha^+, {}^-\beta\}$ is a cut set of $B_{\Gamma^{\times}}$. Then the induced quotient algebra $A=B/ \langle\mathcal{C}\rangle $ of  $B$ has the following quiver 

\begin{figure}[H]
   \adjustbox{scale=.9,center}
{
    \begin{tikzcd}
        a_1^+ \arrow[r,"{}^+\beta"] & a_2^*\arrow[r]& a_3^*\arrow[r]& \dots \arrow[r] & a_{k-1}^*\arrow[r] & a_k^*\arrow[r, "\alpha^-"]& a_1^-
    \end{tikzcd}
}
\end{figure}

where each $a_i^*$ has attached an iterated tilted algebra of Dynkin type $\mathbb{A}_n$. Then by \cite[Section 7, Main theorem, case D3]{AS88} and \cite{Keller91} $A$ is an iterated tilted algebra of Dynkin type $\mathbb{D}$.\end{proof}

\section{A geometric interpretation using dissections}

In this section we use generalized dissections of orbifold surfaces, that were originally considered in \cite{LSV, AB} to study the derived categories of skew-gentle algebras, to give a geometric version of the good cut set. We follow \cite{OPS18,LSV,AB} for the next definitions and results.

\subsection{Computing trivial extensions of skew-gentle algebras using dissections}\label{sec surface disetions}

In this subsection we obtain a skew gentle algebra $A=A(Q,I,\mathrm{Sp})$ from an orbifold dissection. By our previous results, we know that $T(A^{\mathrm{sg}})$ is isomorphic to a skew-Brauer graph algebra $B_{\Gamma^{\times}}$. We use the geometry to compute the quiver $Q_{\Gamma^{\times}}$ and to determine the special cycles, consequently we have all the information needed to define $B_{\Gamma^{\times}}$.

An \emph{orbifold marked surface} is a tuple $O=(S,M,P, \mathcal{O})$ where $S$ is a compact and oriented Riemann surface with boundary, $M$ is a finite set of marked points in the boundary $\partial S$, $P$ is a finite set of marked points in the interior of $S$, called punctures, and $\mathcal{O}$ (with $\mathcal{O} \cap P = \emptyset$) is a set of orbifold points of order two. In particular, when $\mathcal{O}=\emptyset$, we call it a marked surface and denote it simply by $(S,M,P)$. Even more, a marked surface $(S,M,P)$ is unpunctured if $P=\emptyset$.

A \emph{boundary segment} is a curve in $\partial S$ that is delimited by its endpoints in $M$ and does not contain other points in $M$. An \emph{arc} is a curve $\gamma \colon [0,1] \to S$, injective on $(0,1)$ up to orbi-homotopy, considered up to injective orbi-homotopy relative to its endpoints $\gamma(0)$ and $\gamma (1)$, and such that its endpoints belong to $M \cup P \cup \mathcal{O}$. See \cite{CG16} for the definition of orbi-homotopy.

A $\textcolor{red}\times$-dissection of an orbifold marked surface $O=(S, M,P, \mathcal{O})$ is a set $D^{\times}$ of arcs that cuts $O$ into (possibly degenerate) polygons with exactly one edge as a boundary segment, each $x\in\mathcal{O}$ is an endpoint of exactly one arc, there are no boundary components without marked points in the interior of the polygons\footnote{While boundary components without marked points in the interior of polygons can happen in the dissections studied in \cite{LSV}, they will not occur in our dissections.}, and each point $x\in M\cup P$ is the ending point of at least one arc of $D^{\times}$. In this context, a degenerate polygon is one that can contain:
\begin{enumerate}
    \item Special arcs: Arcs ending at an orbifold point $x\in\mathcal{O}$ (with no other arc incident to $x$), or
    \item Self-folded arcs: Arcs ending at a puncture $p\in P$ (with no other arc incident to $p$).\footnote{We take the name from their resemblance to the arcs appearing in self-folded triangles in \cite[Section 2]{fomin2008}.}
\end{enumerate}

The arcs that are not of the form $(1)$ or $(2)$, and the boundary segments, split the surface into topological polygons. To distinguish a special arc from the other arcs, we will denote it as $\gamma^{\times}$. The tuple $(S,M, P, \orbi, D^{\times})$ is called \emph{orbifold dissection}.

The list of polygons might contain what we call \emph{trivial polygons}, i.e. bigons enclosed by an edge in the dissection and one boundary segment and containing no orbifold points, consequently having no special arc, in their interior.

If there are no orbifold points in $O$, we call it $\textcolor{red}\bullet$-dissection of $(S, M, P)$ and the tuple $(S,M,P, D^{\times})$ is called \emph{surface dissection}.

The \textcolor{red}{$\times$}-dissection $(S, M,P, \mathcal{O}, D^{\times})$ is defined in \cite[Sect. 4]{LSV} as the dual dissection of another orbifold dissection. In this construction, the orbifold dissection $(S, M,P, \mathcal{O}, D^{\times})$ has two different kinds of points: the red points and the black points, in our definition, we are omitting the latter. The authors also show that \textcolor{red}{$\times$}-dissection is related to the Koszul dual algebra of $A$. In this section, we establish a different convention to construct a quiver $Q$ and an ideal $I$ that recovers, in our case, a skew-gentle algebra $A(Q,I,\textrm{Sp})$ and also data of the trivial extension $T(A^{\mathrm{sg}})$ of $A^{\mathrm{sg}}$.

We define the quiver of the skew-gentle algebra $A =A(Q,I, \mathrm{Sp})$ from a \textcolor{red}{$\times$}-dissection $(S,M,P, \orbi, D^{\times})$. We start by defining $Q$ and $\mathrm{Sp}$ following the rules: 
\begin{itemize}
    \item[(r1)] Each arc $\gamma_i$ gives rise to a vertex $i$ of $Q$ and each special arc $\gamma_j^{\times}$ corresponds to special vertex $j^{\times}$ of $Q$.
    \item[(r2)]  Each internal angle between arcs (either special or not) in a polygon gives rise to an arrow oriented in counterclockwise direction. 
    
    \item[(r3)] Each special arc $\gamma^{\times}_i$ defines a special loop $f_i:i \to i$ in $Q$.
    
    \item [(r4)] Each self-folded arc $\gamma_j$ defines a non-special loop $g_j:j\to j$ in $Q$.

\end{itemize}

\begin{exmp}\label{ex-facil} In Figure \ref{fig: ejemploA} we show a $\textcolor{red}\bullet$-dissection and a \textcolor{red}{$\times$}-dissection, each one with two degenerate polygons. 

The surface on the left has two self-folded arcs $\gamma_1$ and $\gamma_3$, while the orbifold surface in the middle has a self-folded arc $\gamma_3$ and an special arc $\gamma_1^{\times}$.

\begin{figure}[h]
    \centering
\begin{tikzpicture}[scale=0.8]
   \draw (-5,0) circle (1.5cm);
    \filldraw[red] (-6,0) circle (2pt);
    \filldraw[red] (-4,0) circle (2pt);
    \filldraw[red] (-5,1.5) circle (2pt);
    \filldraw[red] (-5,-1.5) circle (2pt);
    \filldraw[red] (-5,0) circle (2pt);
    \draw[red] (-6,0)--(-5,0) node[pos=0.5, below]{$\gamma_3$};
    \draw[red] (-5,0)--(-4,0) node[pos=0.5, above]{$\gamma_1$};
    \draw[red] (-5,0)--(-5,1.5) node[pos=0.5, left]{$\gamma_2$};
    \draw[red] (-5,0)--(-5,-1.5) node[pos=0.5, right]{$\gamma_4$};
    \draw (0,0) circle (1.5cm);
    \filldraw[red] (-1,0) circle (2pt);
    \filldraw[red] (0,0) circle (2pt);
    \node at (1,0){\textcolor{red}{$\times$}};
    \filldraw[red] (0,-1.5) circle (2pt);
    \filldraw[red] (0,1.5) circle (2pt);
     \draw[red] (-1,0)--(0,0) node[pos=0.5, below]{$\gamma_3$};
    \draw[red] (0,0)--(1,0) node[pos=0.5, above]{$\gamma_1^{\times}$};
    \draw[red] (0,0)--(0,1.5) node[pos=0.5, left]{$\gamma_2$};
    \draw[red] (0,0)--(0,-1.5) node[pos=0.5, right]{$\gamma_4$};
    \node(2) at (4,-.75) {2};
    \node(3) at (5.5,-.75) {3};
    \node(4) at (7,-.75) {4};
    \node(1) at (5.5,.75) {1};
    \node at (4.7,0.25){$\alpha$};
    \node at (4.65,-1){$\beta$};
    \node at (5.85,-1.5){$\delta$};
    \node at (4.9,1.05){$\varepsilon$};
    \node at (6.25,0.25){$\theta$};
    \node at (6.3,-1){$\gamma$};
    \draw[->] (2) --(3);
    \draw[->] (3) --(4);
    \draw[->] (4) --(1);
    \draw[->] (1) --(2);
    \draw [->] (1.west)arc(240:-40:10pt);
    \draw [->] (5.3, -0.85)arc(110:420:10pt);
\end{tikzpicture}
\caption{Two dissections for the disk that define the same $Q$ but give rise to different relations.}
\label{fig: ejemploA}
\end{figure}
\end{exmp}

For each non-trivial (degenerate) polygon $P_i$ denote by $p_i =\alpha_1 \ldots\alpha_t$ the composition of all arrows, such that for all $i$ $\alpha_i \notin \mathrm{Sp}$, arising from the interior of $P_i$ touring the interior of the polygon clockwise.

For the $\textcolor{red}\bullet$-dissection in Example~\ref{ex-facil} (left) $p_1 = \beta \delta \gamma$ and $p_2 = \theta \varepsilon \alpha$. While for the $\textcolor{red}\times$-dissection in Example~\ref{ex-facil} (center) $p_1 = \beta \delta \gamma$ and $p_2 = \theta \alpha$, and $\varepsilon \in \mathrm{Sp}$.

We complete the definition for the skew-gentle algebra $A=A(Q,I, \mathrm{Sp})$ arising from a $\textcolor{red}\times$-dissection by giving the set of relations of $I$.

\begin{enumerate}
    \item  Any composition of arrows $\alpha \beta$ (where $ \alpha, \beta \notin \mathrm{Sp}$) that is not a subpath of a $p_i$ is a relation. In particular, a non-special loop $g_i$ defines a relation $g_i^2=0$.
    \item The composition of two arrows $\alpha \beta$ that is a subpath of $p_i$ and such that $t(\alpha) = s(\beta)$ is the source of a special loop.
\end{enumerate}

 For each path $p_i$, if $s(p_i)$ (or $t(p_i)$) is a source for a special loop $f_j$, define $\widetilde{p_i}= f_j p_i$ (or $p_i f_j$).
With these relations we have that $\widetilde{p_i}$ is an $sp$-maximal path for $A=(Q,I, \mathrm{Sp})$ as in subsection~\ref{subsec the trivial extension}.

\begin{rem} For the skew-gentle algebra arising from an orbifold dissection there is a correspondence between the length of a maximal path modulo $I$ and longest (oriented) sequence of non-boundary edges in a polygon. For this reason each polygon in an orbifold dissection must contain a boundary edge. An absence of the boundary segment would result in path of infinite length modulo $I$. Then the bound quiver algebra  would not be finite dimensional and thus falls beyond our setup.
\end{rem}

For the $\textcolor{red}\times$-dissection in Example~\ref{ex-facil} (center) we have $I = \langle \alpha \beta, \delta^2, \beta \gamma , \gamma \theta, \theta \alpha \rangle$.

Note that it can happen that two arrows form a subpath in $Q$ and they arise from the interior of the same polygon, but their composition has to be a relation. This is the case for $\beta$ and $\gamma$ in Example~\ref{example-torus}.

We can use the \textcolor{red}{$\times$}-dissection to obtain the quiver $Q_{\Gamma^{\times}}$ that will define the (admissible) trivial extension of $A$. Recall that by Theorem~\ref{thm: equiv1} the trivial extension of $A^{\mathrm{sg}}$ is isomorphic to a skew-Brauer graph algebra, with skew-Brauer graph $\Gamma^{\times}$ (having multiplicity function $\mf=1$). 

\begin{figure}[ht]
\centering
\begin{tikzpicture}[scale=0.7]
    \filldraw[gray!30] (0,0) circle (10pt);
    \draw (0,0) circle (10pt);
        \filldraw[gray!30] (6,0) circle (10pt);
    \draw (6,0) circle (10pt);
        \filldraw[gray!30] (0,4) circle (10pt);
    \draw (0,4) circle (10pt);
        \filldraw[gray!30] (6,4) circle (10pt);
    \draw (6,4) circle (10pt);
    \filldraw[red] (0,0.35) circle (2pt);
    \filldraw[red] (0,-0.35) circle (2pt);
    \filldraw[red] (6,0.35) circle (2pt);
    \filldraw[red] (6,-0.35) circle (2pt);
    \filldraw[red] (6,4.35) circle (2pt);
    \filldraw[red] (6,3.65) circle (2pt);
    \filldraw[red] (0,4.35) circle (2pt);
    \filldraw[red] (0,3.65) circle (2pt);
    \draw[red] (0,0.35)--(0,3.65) node[pos=0.5, left]{5};
    \draw[red] (6,0.35)--(6,3.65) node[pos=0.5, right]{5};
    \filldraw[red] (3,4) circle (2pt);
    \filldraw[red] (3,0) circle (2pt);
    \draw[red] (0,4.35) to [out=20, in=150] (3,4);
    \draw[red] (3,4) to [out=-20, in=190] (6,3.65);
     \draw[red] (0,0.35) to [out=20, in=150] (3,0);
    \draw[red] (3,0) to [out=-20, in=190] (6,-0.35);
    \draw[red] (6,3.65) to [out=210, in=190] (6,-0.35);
    \node at (4,2.5){\textcolor{red}{$\times$}};
    \draw[red] (4,2.5) -- (6,3.65) node[pos=0.25, right]{4};
    \node at (1.6,4.8){\textcolor{red}{1}};
    \node at (1.6,0.8){\textcolor{red}{1}};
    \node at (4.4,-0.1){\textcolor{red}{2}};
    \node at (4.4,3.9){\textcolor{red}{2}};
    \node at (4.8,2){\textcolor{red}{3}};
    \node (3) at (8,2){3};
    \node (2) at (9.5,2){2};
    \node (1) at (11,2){1};
    \node (4) at (9.5,3.5){4};
    \node (5) at (9.5,0.5){5};
    \draw [->] (4.west)arc(240:-30:10pt)node[pos=0.8, right]{$f_4$};
    \draw[->] (3)--(2);
    \draw[->] (2) to [out=30, in=150] (1);
    \node at (8.75,1.75){$\alpha$};
    \node at (10.25,2.5){$\beta$};
    \draw[->] (1) to [out=-140, in=-40] (2);
    \node at (10.25,1.8){$\gamma$};
    \draw[->] (2)--(4) node[pos=0.5, right]{$\delta$};
    \draw[->] (4)--(3) node[pos=0.4, left]{$\varepsilon$};
    \draw[->] (3)--(5) node[pos=0.6, left]{$\lambda$};
    \draw[->] (1)--(5)node[pos=0.6, right]{$\omega$};
\end{tikzpicture}
\caption{A \textcolor{red}{$\times$}-dissection of the torus with one boundary component and $Q$ defined by this dissection.}
\label{fig: ejemplo sec 5}
\end{figure}

The quiver $Q_{\Gamma^{\times}}$ is obtained from the \textcolor{red}{$\times$}-dissection by adding to $Q$ a new arrow $\beta_{p_j}$ from $t(p_j)$ to $s(p_j)$ for each path $p_j$ arising from a non-trivial polygon $P_j$. That is, we complete a cycle for each $sp$-maximal path. Then we remove special loops $f_i$ and decorate the corresponding vertices $i^{\times}$. 

The family of special cycles $\mathcal{C}$ is formed by all (the cyclic permutations of) cycles $C_{P_j} = p_j \beta_{p_j}$, arising from non-trivial (degenerate) polygons. This is enough to define the admissible trivial extension $B_{\Gamma^{\times}}$. We see the procedure in the following example.

\begin{exmp} \label{example-torus} The \textcolor{red}{$\times$}-dissection in Figure \ref{fig: ejemplo sec 5} is formed by two polygons. The polygon $P_1$ is delimited by edges $1,2,3,2,1,5$ and boundary, and the polygon $P_2$ is delimited by edges $3,5$ and boundary. The quiver $Q$ for the \textcolor{red}{$\times$}-dissection is the one in Figure~\ref{fig: ejemplo sec 5} (left). In this case $P_1$ contains the maximal path $p_1 = \gamma \delta \varepsilon \alpha \beta \omega$ and generates a new arrow $\beta_{p_1}$ (in red) and $P_2$ contains the maximal path $ p_2 = \lambda$ and generates the new arrow $\beta_{p_2}$ (in blue). For this example, the quivers $Q_{\Gamma^{\times}}$ and $Q^{\mathrm{sg}}_{\Gamma^{\times}}$ are given in Figure~\ref{fig: ejemploB}.

\begin{figure}[ht]
\begin{center}
    \begin{tikzcd}
& 4^{\times} \arrow[ld, swap,"\varepsilon",red]  &  \\
3 \arrow[r,"\alpha",red] \arrow[rd,"\lambda", blue] & 
2 \arrow[r, bend left, "\beta", red] \arrow[u,"\delta", red] & 1 \arrow[l, bend left,swap,"\gamma", red] 
\arrow[ld,"\omega",red] \\
 & 5 \arrow[lu, bend left=40, "\beta_{p_2}", blue] \arrow[ru, bend right=40, swap,"\beta_{p_1}",red]&         
 \end{tikzcd} 
 \hspace{18pt}
 \begin{tikzcd}
& 4^+ \arrow[ld, "{}^+\varepsilon "'] && \\
3 \arrow[rr,bend right, "\alpha"] \arrow[rrd, "\lambda", bend right=20] & 4^- \arrow[l, "^{-}\varepsilon"']   & 2 \arrow[lu, "\delta^+"'] \arrow[l, "\delta^-"'] \arrow[r, "\beta", bend left] & 1 \arrow[ld, "\omega",bend left=5] \arrow[l, "\gamma", bend left] \\ && 5 \arrow[llu, "\beta_{p_2}",bend left=45] \arrow[ru, "\beta_{p_1}", swap,bend right=40] &          
\end{tikzcd}
 \end{center}
\caption{Quiver $Q_{\Gamma^{\times}}$ for a \textcolor{red}{$\times$}-dissection and quiver $Q^{\mathrm{sg}}_{\Gamma^{\times}}$}
\label{fig: ejemploB}
\end{figure}

The trivial extension for $A=A(Q,I, \mathrm{Sp})$ defined by the \textcolor{red}{$\times$}-dissection is $B_{\Gamma^{\times}} =KQ^{\mathrm{sg}}_{\Gamma^{\times}}/I_{\Gamma^{\times}}$, where  $I_{\Gamma^{\times}}$ arises from the cycles $C_{P_1}=p_1 \beta_1$  and $C_{P_2}=p_2 \beta_{p_2}$ and considering the skew-Brauer relations (see Section \ref{subsec-skew-brauer algebra}). In the next list of relations $* \in \{ +,-\}$.

Relation of type 0: $\delta^+ {}^{+}\varepsilon - \delta^- {}^{-}\varepsilon$.

Relations of type I. These relations arise from arcs shared by non-trivial polygons.
\begin{itemize}
    \item arc $3$ sits between $P_1$ and $P_2$: $\alpha \beta \omega\beta_{p_1} \gamma \delta^* {}^{*}\varepsilon - \lambda \beta_{p_2}$
    \item arc $5$, between $P_1$ and $P_2$: $\beta_{p_2} \lambda - \beta_{p_1} \gamma \delta^* {}^*\varepsilon \alpha \beta \omega$
    \item arc $1$, between $P_1$ and $P_1$: $\omega \beta_{p_1} \gamma \delta^* {}^*\varepsilon \alpha \beta - \gamma \delta^* {}^*\varepsilon \alpha \beta \omega \beta_{p_1}$
    \item arc $2$, between $P_1$ and $P_1$: $\delta^* {}^*\varepsilon \alpha \beta \omega \beta_{p_1} \gamma - \beta \omega \beta_{p_1} \gamma \delta^* {}^*\varepsilon \alpha  $
\end{itemize}
Relations of type IIa.
\begin{multicols}{3}
    \begin{itemize}
        \item $\lambda \beta_{p_2} \lambda$
        \item $\beta_{p_2} \lambda \beta_{p_2} $
        \item $ \gamma \delta^* {}^*\varepsilon \alpha \beta \omega \beta_{p_1} \gamma$
        \item $\delta^* {}^*\varepsilon \alpha \beta \omega \beta_{p_1} \gamma \delta^*$
        \item $\alpha \beta \omega \beta_{p_1} \gamma \delta^* {}^*\varepsilon \alpha  $
        \item $ \beta \omega \beta_{p_1} \gamma \delta^* {}^*\varepsilon \alpha  \beta $
        \item $\omega \beta_{p_1} \gamma \delta^* {}^*\varepsilon \alpha  \beta \omega$
        \item $\beta_{p_1} \gamma \delta^* {}^*\varepsilon \alpha  \beta \omega \beta_{p_1}$
    \end{itemize}
    \end{multicols}
Relations of type IIb: $\delta^- {}^{-}\varepsilon \alpha \beta \omega \beta_{p_1} \gamma \delta^{+}$ and $\delta^+ {}^{+}\varepsilon \alpha \beta \omega \beta_{p_1} \gamma \delta^{-}$.

Relations of type III: 
$ \beta \gamma, \gamma \beta, \alpha \delta^*, {}^{*}\varepsilon \lambda$, $\beta_{p_1} \omega$, $\beta_{p_2} \alpha$, $\lambda \beta_{p_1}$, $\omega \beta_{p_2}$.

 \begin{figure}[h]
 \centering
 \begin{tikzpicture}[scale=0.8]
    \node at (1.3,0.6) {$\times$};
    \draw (0,0) to  node[pos=0.5, above]{\tiny{4}}(1.3,0.6);
    \draw[rotate=30] (-0.5,0) ellipse (0.5cm and 1cm);
    \draw (0,0) to  node[pos=0.5, above]{\tiny{3}}(3,0);
    \draw (0,0) to[out=80,in=170] (2.7,1.3);  
    \draw (2.7,-1)to[out=-170,in=-60] (0,0); 
    \draw (2.7,1.3) to[out=-10, in=90] (3.5, 0);
    \draw (2.7,-1 )to[out=10, in=-90] (3.5,0);
    \filldraw[white] (0.2,-0.95) circle(3pt);
     \filldraw[white] (3.3,-0.65) circle(3pt);
    \draw (0,0) to[out=-120, in=-150] (2.8,-1.5) to[out=30, in=-10] (3,0);
    \filldraw[red] (0,0) circle(3pt);
    \filldraw[blue] (3,0) circle(3pt);
    \node at (1.7,1.2){\tiny{2}};
    \node at (1.7,-1.8){\tiny{5}};
    \node at (-0.8,-0.8){\tiny{1}};
 \end{tikzpicture}
 \caption{Skew-Brauer graph.}
 \label{fig: brauer sec 5}
 \end{figure}
\end{exmp}
The skew-Brauer graph for Example \ref{example-torus} is depicted in Figure \ref{fig: brauer sec 5}.

We summarize the construction proposed in this subsection in the next proposition.

\begin{prop}\label{prop: geoemtric definition} 
Each \textcolor{red}{$\times$}-dissection $(S,M, P, \orbi, D^{\times})$ defines a skew-gentle algebra $A=\SG$. Moreover, the quiver and relations for its trivial extension $T(A^{\mathrm{sg}})$ can be obtained from the dissection. 
\end{prop}

\begin{proof}
    The \textcolor{red}{$\times$}-dissection describes a skew-gentle algebra $A=\SG$. The ideal $I$ is given so that each (non-trivial) polygon is associated to a $sp$-maximal path $\widetilde{p_i}$. We define a new arrow $\beta_{p_i}$ for each of these paths as in Theorem \ref{teo:ideal relaciones}, we set the $i^{\times}$ as special vertices and we define the special cycles (up to cyclic permutation) as $p_i \beta_{p_i}$. With all this we get the quiver and relations for $T(A^{\mathrm{sg}}) = B_{\Gamma^{\times}_A}$.
\end{proof}

\subsection{Contracting and adding boundary segments.}\label{subsec:good cuts}

Let $\mathcal{X} = (S,M, P, \orbi, D^{\times})$ be a \textcolor{red}{$\times$}-dissection that splits the surface  into (degenerate) polygons, where $P_1, \dots, P_n$ denote the non-trivial polygons.
We provide a geometric characterization of good cut sets for the trivial extension of skew-gentle algebras by defining a two-step operation.

For an arbitrary fixed non-trivial polygon $P_i$ and either a fixed internal angle $\theta$ between two arcs of $P_i$ or a fixed puncture $x$ that is the ending point of a self-folded arc of $P_i$ (and no other arcs), we define the addition-contraction at $(P_i,\theta)$ or $(P_i, x)$, respectively, as follows.

\textbf{(Step 1: Contraction)} For the polygon $P_i$, there is exactly one edge, say $b$, that is a boundary segment. We note that $b$ can be a boundary component with one marked point. If $b$ is a boundary segment with two different endpoints, then we contract $b$ by identifying the points and obtaining a vertex $\widetilde{x}$. If $b$ is a boundary component with one marked point, then we contract $b$ to a single marked point $\widetilde{x}$. In both cases, we obtain a polygon $\widetilde{P}_i$ with no boundary segment as an edge.

\textbf{(Step 2: Addition for $(P_i,\theta)$)} For the internal angle $\theta$ in $\widetilde{P}_i$ opposite to a vertex $y$ of $\widetilde{P_i}$, we add a boundary (segment or component) as follows.

\begin{enumerate}
    \item if $y\in M$, we add a new boundary segment, splitting the red marked point (see Figure~\ref{Fig:add-contrac2} left);
    \item when $y\in P$, we add a new boundary component in $\theta$ having $y$ as a marked point (see Figure~\ref{Fig:add-contrac2} center).
  
\end{enumerate}

\textbf{(Step 2: Addition for $(P_i,p)$)} If $x\in P$ is the ending point of a self-folded arc of $P_i$ (and no other arcs), we replace $x$ by a boundary component, having $x$ has a marked point (Figure~\ref{Fig:add-contrac2} right).

\begin{figure}[H]
\centering
\begin{tikzpicture}[scale=0.80]
\draw (-1.7,1.2) to (-1.7, -2);
\draw (-1.7,1.2) to (17.5,1.2);
\draw (17.5,1.2) to (17.5,-2);
\draw (-1.7,-2) to (17.5,-2);
        \filldraw (0,0) circle (1pt) node[above] {\tiny$y$};
        \filldraw (-1,-0.5) circle (1pt);
        \filldraw (1,-0.5) circle (1pt);
        \draw[dashed] (-1.5,0.45) to [bend right] (1.5,0.45);
        \draw (-.1,-.15) to[bend right] (.1,-.15);
        \draw (0,0) to[bend left] (-1,-0.5);
        \draw (0,0) to[bend right] (1,-0.5);
        \draw (-1,-0.5) to[bend left] (-0.7, -1.3);
        \draw (1,-0.5) to[bend right] (0.7,-1.3);
        \node at (0,-1.5) {$\dots$};
        \node at (-0.5,-0.7) {\tiny$\widetilde{P_i}$};
       \node at (0, -.4) {\tiny$\theta$};
   \filldraw (3,0.2) circle (1pt) node[above] {\tiny$y$};
        \filldraw (4.5,0.2) circle (1pt) node[above] {\tiny$y'$};
        \filldraw (4.7,-0.6) circle (1pt);
        \filldraw (2.6,-0.6) circle (1pt);
        \draw[dashed] (2.3,0.26) ..controls (3.5, 0.1).. (5,0.26);
        \draw (3,0.2) to[bend left] (2.6,-0.6);
        \draw (2.6,-0.6) to[bend left] (3.1,-1.3);
        \draw (4.5,0.2) to[bend right] (4.7, -0.6);
        \draw (4.7,-0.6) to[bend right] (4.2,-1.3);
        \node at (3.7,-0.6) {\tiny$P_i'$};
       \node at (3.7,-1.3) {\tiny$\dots$};
       \draw (5.2,1.2) to (5.2, -2);
        \filldraw (7,0) circle (1pt) node[above] {\tiny$y$};
        \filldraw (6,-0.5) circle (1pt);
        \filldraw (8,-0.5) circle (1pt);
        \node at (7, -.4) {\tiny$\theta$};
        \draw (6.9,-.15) to[bend right] (7.1,-.15);
        \draw (7,0) to[bend left] (6,0.2);
       \draw (7,0) to[bend right] (8,0.2);
        \draw (7,0) to[bend left] (6,-0.5);
        \draw (7,0) to[bend right] (8,-0.5);
        \draw (6,-0.5) to[bend left] (6.3, -1.3);
        \draw (8,-0.5) to[bend right] (7.7,-1.3);
        \node at (7,-1.5) {$\dots$};
        \node at (6.5,-0.7) {\tiny$\widetilde{P_i}$};
        \node at (6.2,-0.2) {\tiny$\vdots$};
        \node at (7.8,-0.2) {\tiny$\vdots$};
        \filldraw (10,0) circle (1pt) node[above] {\tiny$y$};
        \filldraw (9,-0.5) circle (1pt);
        \filldraw (11,-0.5) circle (1pt);
        \draw (10,0) to[bend left] (9,0.2);
       \draw (10,0) to[bend right] (11,0.2);
        \draw (10,0) to[bend left] (9,-0.5);
        \draw (10,0) to[bend right] (11,-0.5);
        \draw (9,-0.5) to[bend left] (9.3, -1.3);
        \draw (11,-0.5) to[bend right] (10.7,-1.3);
        \draw[pattern=north west lines, pattern color=gray, densely dashed] (10,-.35) ellipse (0.1cm and 0.3 cm);
        \node at (9.5,-0.7) {\tiny${P_i}'$};
        \node at (9.2,-0.2) {\tiny$\vdots$};
        \node at (10.8,-0.2) {\tiny$\vdots$};
    \draw (11.5,1.2) to (11.5, -2);
        \filldraw (16,0) circle (1pt);
        \filldraw (15,-0.5) circle (1pt);
        \filldraw (17,-0.5) circle (1pt);
        \draw (16,0) to[bend left] (15,0.2);
       \draw (16,0) to[bend right] (17,0.2);
       \draw (16,0) to[bend left] (15,-0.5);
        \draw (16,0) to[bend right] (17,-0.5);
        \draw (15,-0.5) to[bend left] (15.3, -1.3);
       \draw (17,-0.5) to[bend right] (16.7,-1.3);
        \node at (16,-1.5) {$\dots$};
        \draw (16,0) to (16,-0.8);
       \filldraw (16,-0.8) circle (1pt) node[right] {\tiny$x$};
        \node at (15.5,-0.7) {\tiny${P_i}'$};
        \draw[pattern=north west lines, pattern color=gray, densely dashed] (16,-1) circle (0.2cm);
        \node at (15.2,-0.2) {\tiny$\vdots$};
        \node at (16.8,-0.2) {\tiny$\vdots$};
        \filldraw (13,0) circle (1pt);
        \filldraw (12,-0.5) circle (1pt);
        \filldraw (14,-0.5) circle (1pt);
       \draw (13,0) to (13,-0.8);
       \filldraw (13,-0.8) circle (1pt) node[below] {\tiny$x$};
        \draw (13,0) to[bend left] (12,0.2);
       \draw (13,0) to[bend right] (14,0.2);
        \draw (13,0) to[bend left] (12,-0.5);
        \draw (13,0) to[bend right] (14,-0.5);
        \draw (12,-0.5) to[bend left] (12.3, -1.3);
        \draw (14,-0.5) to[bend right] (13.7,-1.3);
        \node at (12.5,-0.7) {\tiny$\widetilde{P_i}$};
        \node at (12.2,-0.2) {\tiny$\vdots$};
        \node at (13.8,-0.2) {\tiny$\vdots$};
        \end{tikzpicture}
\caption{Adding a boundary segment at $\theta$ or $x$.}
\label{Fig:add-contrac2}
\end{figure}

In the two cases, we obtain a new polygon $P_i'$ with exactly one edge that is a boundary segment.
We denote the resulting \textcolor{red}{$\times$}-dissection by: $\tau_{(P_i, \theta)}\mathcal{X}$ or $\tau_{(P_i, x)}\mathcal{X}$, respectively. The new dissection splits $S$ into (degenerate) polygons, and the set of non-trivial polygons is $\{P_1, \dots, P_m\}\setminus P_i\cup \{P_i'\}$.

\begin{exmp} A \textcolor{red}{$\times$}-dissection for the annulus and the new \textcolor{red}{$\times$}-dissection after a contraction-addition at $(P_2,\theta)$.
\begin{figure}[h!]
\centering
\begin{tikzpicture}[scale=0.8]
    \draw (-5,0) circle (2cm);
    \filldraw[lightgray] (-5,0) circle (0.2cm);
    \draw (-5,0) circle (0.2cm);
    \draw[red] (-5,-0.2) ellipse (0.3cm and 0.4 cm);
    \filldraw[red] (-7,0) circle (2pt);
    \filldraw[red] (-5,0.2) circle (2pt);
    \filldraw[red] (-3, 0.3) circle (2pt);
    \node at (-4,1.2){\textcolor{red}{\tiny{$\times$}}};
    \node at (-4.3,0.7){\textcolor{red}{\tiny{$\times$}}};
    \draw[red] (-5,0.2) to[bend right] (-7,0);
    \draw[red] (-5,0.2) to[bend right] (-3,0.3);
    \draw[red] (-3,0.3) to (-4,1.2);
    \draw[red] (-3,0.3) to (-4.3,0.7);
    \node at (-3.9,-0.2){\tiny{3}};
    \node at (-3.9,0.4){\tiny{2}};
    \node at (-3.9,0.9){\tiny{1}};
    \node at (-5.9,0.2){\tiny{4}};
    \node at (-5,-0.8){\tiny{5}};
    \node at (-5,-1.5){\tiny{$P_1$}};
    \node at (-5,1){\tiny{$P_2$}};
    \draw[dashed] (-5.3,0.3) to [out=40, in=100] (-4.65,0.1);
    \node at (-5, 0.6){\tiny{$\theta$}};
    \draw[red] (0,-0.3) ellipse (0.3cm and 0.4 cm);
    \draw[red] (0,0.95) ellipse (0.6cm and 1.05 cm);
    \draw (0,0) circle (2cm);
    \filldraw[lightgray] (0,0) circle (0.2cm);
    \draw (0,0) circle (0.2cm);
    \filldraw[red] (-0.2,0) circle (2pt);
    \filldraw[red] (0.2,0) circle (2pt);
    \filldraw[red] (0, 2) circle (2pt);
    \node at (-0.2,1){\textcolor{red}{\tiny{$\times$}}};
    \node at (0.2,1){\textcolor{red}{\tiny{$\times$}}};
    \draw[red] (-0.2,1) to (0,2);
    \draw[red] (0.2,1) to (0,2);
    \node at (-0.2,1.4){\tiny{1}};
    \node at (0.2,1.4){\tiny{2}};
    \node at (0.8,1){\tiny{3}};
    \node at (-0.8,1){\tiny{4}};
    \node at (0,-0.8){\tiny{5}};
    \node at (0,-1.5){\tiny{$P_1$}};
    \node at (0,0.4){\tiny{$P_2'$}};
\end{tikzpicture}
{\begin{tikzcd}
2  \arrow[loop left, "f_1"] \arrow[r, "\alpha_2"] & 3 \arrow[d, "\alpha_3"]  & 5 \arrow[l, "\beta_1"'] &  & 2 \arrow[r, "\alpha_2"]  \arrow[loop left, "f_1"] & 3   & 5 \arrow[l, "\beta_1"'] \\
1 \arrow[u, "\alpha_1"]  \arrow[loop left, "f_2"] & 4 \arrow[ru, "\beta_3"'] &    &  & 1  \arrow[loop left, "f_2"]\arrow[u, "\alpha_1"] & 4 \arrow[l, "\alpha_4"] \arrow[ru, "\beta_3"'] &         \end{tikzcd}}
\caption{A \textcolor{red}{$\times$}-dissection $\mathcal{X}$ and the new \textcolor{red}{$\times$}-dissection  $\tau_{(P_2, \theta)} \mathcal{X}$.}
\label{fig: ejemplo sec 5 cut}
\end{figure}

Observe that both \textcolor{red}{$\times$}-dissections in Figure~\ref{fig: ejemplo sec 5 cut}, generate the skew-Brauer graph algebra $B$ (with $\mf =1$) from Example \ref{example-1}. 
Each one of two \textcolor{red}{$\times$}-dissections defines a skew-gentle algebra that is a good cut (Definition \ref{def:good cut}) for $B$.
\end{exmp}

\begin{thm}\label{thm:contracting and adding}
Let $A$ be a skew-gentle with $\mathrm{Sp}$ the set of special vertices, and $\mathcal{X}=(S,M, P, \orbi, D^{\times})$ its \textcolor{red}{$\times$}-dissection and let $\mathcal{D}$ be a good cut set of $T(A^{\mathrm{sg}})$. Then there is a sequence of contraction-addition moves taking $\mathcal{X}$ to $\widetilde{\mathcal{X}}$, where $\widetilde{\mathcal{X}}$ is the \textcolor{red}{$\times$}-dissection of the quotient algebra of $T(A^{\mathrm{sg}})$ induced by $\mathcal{D}$.
\end{thm}

\begin{proof}
    Let $\mathcal{D}$ be a good cut of $T(A^{\mathrm{sg}})$. By definition there exists an admissible cut $\mathcal{D}'$  of the trivial extension $T(A')$ where $A'$ is the associated gentle algebra of $A$. By \cite[Theorem 1.3]{Sch15}, $\mathcal{D}'$ is a set of arrows containing exactly one arrow from each one of the special cycles of $T(A')$, which by Lemma~\ref{lem:maximal-paths AyA'} are in bijection with the non-trivial (degenerate) polygons of $\mathcal{X}$. To fix notation, denote by $\alpha_{i}$ the arrow of $\mathcal{D}'$ associated with the non-trivial polygon $P_i$.
    
    Each arrow $\alpha_{i}$ of $\mathcal{D}'$ corresponds to either an internal angle $\theta_i$ of a polygon $P_i$ opposite to a vertex $x\in M\cup P$, or to a puncture $x_i$ that is the ending point of a self-folded arc of $P_i$ (and no other arcs) or corresponds to a boundary segment of $P_i$. Only in the first two cases, we need to apply a contraction-addition at $(P_i,\theta_i$ or at $(P_i, x_i)$.

    Since $\mathcal{D}'$ is an admissible cut of $T(A')$, there are no repeated polygons in the list of contraction-addition at  $(P_i, \theta_i)$ or $(P_i, x_i)$ induced by the arrows of $\mathcal{D}'$. Then we can define a sequence $\tau$ of contraction-addition, and it is well-defined. Finally, by construction $T(A^{\mathrm{sg}})/\langle \mathcal{D\rangle}$ is isomorphic to the skew-gentle algebra obtained from $\tau \mathcal{X}$. \end{proof}

\subsection{On derived equivalent skew-gentle algebras.}\label{subsec:derived-eq}

The following example shows that contraction-addition does not necessarily induces derived equivalent skew-gentle algebras.

\begin{figure}[h]
\centering
\begin{tikzpicture}
    \draw (-5,0) circle (1.5cm);
    \filldraw[lightgray] (-5,0) circle (0.2cm);
    \draw (-5,0) circle (0.2cm);
    \filldraw[red] (-5,-1.5) circle (2pt);
    \filldraw[red] (-5,0.2) circle (2pt);
    \node at (-6,0){\textcolor{red}{$\times$}};
    \draw[red] (-6,0) to[bend right] (-5,-1.5);
    \node at (-5.7,-0.65){\tiny{3}};
    \node at (-5.3,-0.65){\tiny{2}};
    \node at (-4.5,-0.65){\tiny{1}};
    \node at (-5,-0.5){\tiny{$P_1$}};
    \node at (-5,1){\tiny{$P_2$}};
    \draw[dashed] (-5.2,-1.2) to [bend left] (-4.75,-1.2);
    \node at (-5, -1){\tiny{$\theta$}};
    \draw [red] (-5,-1.5) to[out=150, in=150] (-5,0.2);
    \draw [red] (-5,-1.5) to[out=30, in=40] (-5,0.2);
    \draw (0,0) circle (1.5cm);
    \filldraw[red] (220:1.5) circle (2pt);
    \filldraw[red] (0,0) circle (2pt);
    \filldraw[red] (310:1.5) circle (2pt);
    \node at (-0.4,0.7){\textcolor{red}{$\times$}};
    \draw[red] (220:1.5) to[bend left] (-0.4,0.7);
    \draw[red] (220:1.5) to (0,0);
    \draw[red] (310:1.5) to (0,0);
    \node at (-0.8,0.1){\tiny{3}};
    \node at (-0.8,-0.35){\tiny{2}};
    \node at (0.6,-0.35){\tiny{1}};
    \node at (0,-0.6){\tiny{$P_1'$}};
    \node at (0,1){\tiny{$P_2$}};
\end{tikzpicture}
\caption{\textcolor{red}{$\times$}-dissection $\mathcal{X}$ (left) and the new \textcolor{red}{$\times$}-dissection $\tau_{(P_1,\theta)} \mathcal{X}$ (right).}
\label{fig: ejemplo_derived_eq}
\end{figure}

\begin{exmp}
Let $A$ and $B$ the skew-gentle algebras associated to the $\textcolor{red}{\times}$-dissections depicted in Figure~\ref{fig: ejemplo_derived_eq}, with non-trivial polygons $P_1$ and $P_2$ and $P_1'$ and $P_2$, respectively. One can see that $A$ and $B$ are related by an contraction-addition at $(P_1,\theta)$, and even more the \textcolor{red}{$\times$}-dissections are induced by good cuts of a skew-Brauer graph algebra with multiplicity $\mf =1$.

Let $A$ be the skew-gentle algebra defined by $\mathcal{X}$ in Figure \ref{fig: ejemplo_derived_eq}, and consider $B$ defined by $\tau_{(P_1,x)} \mathcal{X}$.

\begin{center}

\begin{tikzcd}
1 \arrow[r, "\beta_1"', bend right] \arrow[r, "\alpha_1", bend left] & 2 \arrow[r, "\alpha_2"] & 3 \arrow["f_3"', loop, distance=2em, in=35, out=325] &  & 1 \arrow[r, "\alpha_1", bend left] & 2 \arrow[l, "\beta_2", bend left] \arrow[r, "\alpha_2"] & 3 \arrow["f_3"', loop, distance=2em, in=35, out=325]
\end{tikzcd}
\end{center}
Hence $A$ is defined by the quiver on the left and taking the ideal $\langle \beta_1 \alpha_2, f_3^2-f_3\rangle$, while $B$ is defined by the quiver on the right considering the ideal $\langle \beta_2 \alpha_1, \alpha_1 \beta_2, f_3^2-f_3\rangle$.

Following \cite[Theorem 7.3]{LSV} and looking at the \textcolor{red}{$\times$}-dissections\footnote{Theorem 7.3 in \cite{LSV} can be re stated $\mathrm{det}\, (C_A)_q = \prod_{k \geq 1} (1- (-q)^k)^{c_k}$ where $c_k$ is the number of red punctures with $k$ incident red arcs.}, it is easy to see that the determinants of their $q$-Cartan matrices are $\mathrm{det}\, (C_A)_q = 1$ and $\mathrm{det}\, (C_B)_q = 1-q^2$. The ordinary Cartan matrices are obtained by evaluating $q=1$, hence we can find their determinants $\mathrm{det}\, (C_A) = 1$ and $\mathrm{det}\, (C_B) = 0$. Since these determinants are derived invariants \cite[Proposition 1.5]{BSk}, we conclude that $A$ and $B$ are not  derived equivalent.
\end{exmp}

\subsection{Reflections as geometric moves on a \textcolor{red}{$\times$}-dissection}
By Theorem~\ref{thm:cuts and reflections} and  Theorem~\ref{thm:contracting and adding}, any sequence of reflections in a skew-gentle algebra $A$ can be interpreted as a sequence of contraction-addition on its \textcolor{red}{$\times$}-dissection. In this section, we present another interpretation of reflections as local moves in \textcolor{red}{$\times$}-dissections. We will present examples of reflections on \textcolor{red}{$\times$}-dissections and show how these moves behave locally. 

Let $A$ be a skew-gentle algebra defined by a \textcolor{red}{$\times$}-dissection. Let $y$ be a vertex that is a source or a target in $Q_{A'}$, where $A'$ is the auxiliary gentle algebra of $A$. Such a vertex will be associated to a red arc $y$ that can be a special arc, as is Figure~\ref{fig: ejemplo_refl} right, or a regular red arc. If the vertex $y$ is a target, the arc $y$ is such that if we pivot counterclockwise on one of its endpoints, we meet the boundary edge of a polygon. See Figure~\ref{fig: ejemplo_refl} left.

Let $y$ be a non-special target vertex. Consider $p_a$ and $p_b$ the maximal paths in $A'$ ending at $y$. See that one of these paths might be trivial, that is the case when one of the polygons having $y$ as an edge is a trivial polygon. Let $\beta_{p_1}$  and $\beta_{p_2}$ the new arrows in $Q_{\widehat{A'}}$, then $S^+_y(A)$ is given by a quiver and the (non-admissible) ideal defined by a new \textcolor{red}{$\times$}-dissection, exchanging locally the vertex $y$ and the vertex $y'$ as in the next figure. If $y$
is special, we see the local exchange $S^-_y (A)$ in the right part of the figure. The negative and positive reflections for non-special and special vertices are analogous. 

\adjustbox{scale=0.9,center}{
\begin{tikzcd}
a_1 \arrow[r] & a_2 \arrow[r] & \cdots \arrow[r]  & a_t \arrow[rd] &   \\
b_1 \arrow[r] & b_2 \arrow[r] & \cdots  \arrow[r] & b_m \arrow[r]  & y
\end{tikzcd} \hspace{10pt} \begin{tikzcd}
y \arrow[loop, distance=1.5em, in=125, out=55] \arrow[r] & a_1 \arrow[r] & \cdots  \arrow[r] & a_t  
\end{tikzcd}
}

\adjustbox{scale=0.9,center}{
\begin{tikzcd}
a_1 \arrow[r] & a_2 \arrow[r] & \cdots \arrow[r]  & a_t &                                                                                   \\
b_1 \arrow[r] & b_2 \arrow[r] & \cdots  \arrow[r] & b_m & y' \arrow[llllu, "\beta_{p_a}"] 
\arrow[llll, "\beta_{p_b}", bend left=15, shift left]\end{tikzcd} \hspace{10pt} \begin{tikzcd}
y' \arrow[loop, distance=1.5em, in=125, out=55]  
& a_1 \arrow[r] & \cdots \arrow[r]  & a_t \arrow[lll, bend left, shift left]
\end{tikzcd}
}

\begin{figure}[h]
\centering
\begin{tikzpicture}
    \filldraw[red] (0,1.5) circle (2pt);
    \filldraw[red] (1.5,1.5) circle (2pt);
    \filldraw[red] (2.5,1.5) circle (2pt);
    \filldraw[red] (4,1.5) circle (2pt);
    \filldraw[red] (5.5,1.5) circle (2pt);
    \filldraw[red] (0,0) circle (2pt);
    \filldraw[red] (1.5,0) circle (2pt);
    \filldraw[red] (3,0) circle (2pt);
    \filldraw[red] (4,0) circle (2pt);
    \filldraw[red] (5.5,0) circle (2pt);
    \draw[red] (0,1.5) to[bend right] (1.5,1.5);
    \draw[black, dashed] (1.5,1.5) to[bend right] (2.5,1.5);
    \draw[red] (2.5,1.5) to[bend right] (4,1.5);
    \draw[red] (4,1.5) to[bend right] (5.5,1.5);
    \draw[red] (0,0) to[bend left] (1.5,0);
    \draw[red] (1.5,0) to[bend left] (3,0);
    \draw[black, dashed] (3,0) to[bend left] (4,0);
    \draw[red] (4,0) to[bend left] (5.5,0);
    \node at (0,0.85){\textcolor{red}{$\vdots$}};
    \node at (5.5,0.85){\textcolor{red}{$\vdots$}};
    \draw[red] (1.5,1.5) to (4,0);
    \node at (0.75,1.5){\tiny{$b_m$}};
    \node at (3.25,1.5){\tiny{$a_1$}};
    \node at (4.75,1.5){\tiny{$a_2$}};
    \node at (0.75,0){\tiny{$b_2$}};
    \node at (2.25,0){\tiny{$b_1$}};
    \node at (4.75,0){\tiny{$a_t$}};
    \node at (3.25,0.6){\tiny{$y$}};
    \draw[->] (1.7, 1.6) to[bend right](2.3,1.6);
    \draw[->] (3.8, -0.2) to[bend right](3.3,-0.2);
    \filldraw[red] (0,-1) circle (2pt);
    \filldraw[red] (1.5,-1) circle (2pt);
    \filldraw[red] (2.5,-1) circle (2pt);
    \filldraw[red] (4,-1) circle (2pt);
    \filldraw[red] (5.5,-1) circle (2pt);
    \filldraw[red] (0,-2.5) circle (2pt);
    \filldraw[red] (1.5,-2.5) circle (2pt);
    \filldraw[red] (3,-2.5) circle (2pt);
    \filldraw[red] (4,-2.5) circle (2pt);
    \filldraw[red] (5.5,-2.5) circle (2pt);
    \draw[red] (0,-1) to[bend right] (1.5,-1);
    \draw[black, dashed] (1.5,-1) to[bend right] (2.5,-1);
    \draw[red] (2.5,-1) to[bend right] (4,-1);
    \draw[red] (4,-1) to[bend right] (5.5,-1);
    \draw[red] (0,-2.5) to[bend left] (1.5,-2.5);
    \draw[red] (1.5,-2.5) to[bend left] (3,-2.5);
    \draw[black, dashed] (3,-2.5) to[bend left] (4,-2.5);
    \draw[red] (4,-2.5) to[bend left] (5.5,-2.5);
    \node at (0,-1.65){\textcolor{red}{$\vdots$}};
    \node at (5.5,-1.65){\textcolor{red}{$\vdots$}};
    \draw[red] (2.5,-1) to (3,-2.5);
    \node at (0.75,-1){\tiny{$b_m$}};
    \node at (3.25,-1){\tiny{$a_1$}};
    \node at (4.75,-1){\tiny{$a_2$}};
    \node at (0.75,-2.5){\tiny{$b_2$}};
    \node at (2.25,-2.5){\tiny{$b_1$}};
    \node at (4.75,-2.5){\tiny{$a_t$}};
    \node at (3,-1.75){\tiny{$y'$}};
    \filldraw[red] (7.5,1.5) circle (2pt);
    \filldraw[red] (9.5,1.5) circle (2pt);
    \filldraw[red] (7.5,0) circle (2pt);
    \filldraw[red] (8.5,-0.5) circle (2pt);
    \filldraw[red] (9.5,0) circle (2pt);
    \node[red] at (8.5, 0.5){$\times$};
    \draw[red] (7.5, 1.5) to[bend left] (7.5,0);
    \node at (8, -0.2){\textcolor{red}{$\ddots$}};
    \draw[red] (8.5, -0.5) to[bend left] (9.5,0);
    \draw[red] (9.5, 0) to[bend left] (9.5,1.5);
    \draw[black, dashed] (7.5,1.5) to[bend right] (9.5,1.5);
    \draw[red] (9.5, 1.5) to (8.5,0.5);
    \draw[->] (8.7, 1.5) to[bend left] (8.2, 1.5);
    \node at (7.5,0.75){\tiny{$a_t$}};
    \node at (9.5,0.75){\tiny{$a_1$}};
    \node at (9,0.75){\tiny{$y$}};
    \node at (9,-0.3){\tiny{$a_2$}};
    \filldraw[red] (7.5,-1) circle (2pt);
    \filldraw[red] (9.5,-1) circle (2pt);
    \filldraw[red] (7.5,-2.5) circle (2pt);
    \filldraw[red] (8.5,-3) circle (2pt);
    \filldraw[red] (9.5,-2.5) circle (2pt);
    \node[red] at (8.5, -2){$\times$};
    \draw[red] (7.5, -1) to[bend left] (7.5,-2.5);
    \node at (8, -2.7){\textcolor{red}{$\ddots$}};
    \draw[red] (8.5, -3) to[bend left] (9.5,-2.5);
    \draw[red] (9.5, -2.5) to[bend left] (9.5,-1);
    \draw[black, dashed] (7.5,-1) to[bend right] (9.5,-1);
    \draw[red] (7.5, -1) to (8.5,-2);
    \node at (7.5,-1.75){\tiny{$a_t$}};
    \node at (9.5,-1.75){\tiny{$a_1$}};
    \node at (8.25,-1.5){\tiny{$y'$}};
    \node at (9,-2.8){\tiny{$a_2$}};
\end{tikzpicture}
\caption{Positive reflection $S_y^+$ with $y$ a non-special target vertex (right). Negative reflection $S_y^-$ with $y$ a special source vertex (left)}
\label{fig: ejemplo_refl}
\end{figure}

\begin{exmp} Let $A$ be the skew-gentle algebra defined by the \textcolor{red}{$\times$}-dissection on the left of Figure~\ref{fig: ejemplo sec 5 cut}. Thus $A= KQ/I$ where $Q$ is the quiver on the left,

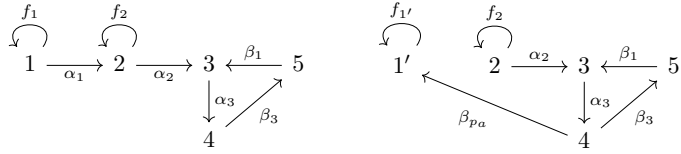
\adjustbox{scale=0.9,center}{
\begin{tikzcd}
1 \arrow[r, "\alpha_1"'] \arrow["f_1"', loop, distance=1.5em, in=125, out=55] & 2 \arrow[r, "\alpha_2"'] \arrow["f_2"', loop, distance=1.5em, in=125, out=55] & 3 \arrow[d, "\alpha_3"]  & 5 \arrow[l, "\beta_1"'] \\
  &     & 4 \arrow[ru, "\beta_3"'] &    \end{tikzcd}  \hspace{15pt} \begin{tikzcd}
1' \arrow["f_{1'}"', loop, distance=1.5em, in=125, out=55] & 2 \arrow["f_2"', loop, distance=1.5em, in=125, out=55] \arrow[r, "\alpha_2"] & 3 \arrow[d, "\alpha_3"]                              & 5 \arrow[l, "\beta_1"'] \\
       & & 4 \arrow[ru, "\beta_3"'] \arrow[llu, swap,"\beta_{p_a}"'] &                        
\end{tikzcd}  }
  and $I = \langle f_1^2 - f_1, f_2^2- f_2, \alpha_1 \alpha_2 , \alpha_3 \beta_3 , \beta_1 \alpha_3 \rangle$. Then $\beta_{p_a}$ is the new arrow arising from the path $p_a=\alpha_1\alpha_2\alpha_3$ and $S_1^-(A)$ is the skew-gentle algebra defined by the quiver on the right and the ideal $J= \langle f_{1'}^2 - f_{1'}, f_2^2- f_2, \alpha_3 \beta_3 , \beta_1 \alpha_3 \rangle$, that is associated to the \textcolor{red}{$\times$}-dissection after applying a move as the one in Figure \ref{fig: ejemplo_refl} right.
\end{exmp}

\bibliographystyle{alpha}
\bibliography{bibliografia.bib}
\end{document}